\documentclass[12pt]{amsart}
\pdfoutput=1
\usepackage{amsmath, amssymb, amsthm, amsbsy, bbm, amsfonts, color, verbatim, array, floatflt, yfonts}

\usepackage[margin=1.2in]{geometry}
\usepackage[all]{xy}
\usepackage[english]{babel} 
\usepackage[shortlabels]{enumitem}
\usepackage{ stmaryrd } 

\usepackage{graphicx, color, overpic}
\usepackage[usenames,dvipsnames,svgnames,table]{xcolor}
\usepackage[pdfpagelabels]{hyperref}
\usepackage{cleveref}

\usepackage{multirow}
\usepackage{easybmat}

\usepackage{thmtools}
\usepackage{thm-restate} 

\theoremstyle{plain}
\newtheorem{thm}{Theorem}[section]
\newtheorem*{thm*}{Theorem}
\newtheorem{prop}[thm]{Proposition}
\newtheorem*{prop*}{Proposition}
\newtheorem{lemma}[thm]{Lemma}
\newtheorem*{lemma*}{Lemma}
\newtheorem{corollary}[thm]{Corollary}

\theoremstyle{definition}
\newtheorem{definition}[thm]{Definition}
\newtheorem{example}[thm]{Example}

\newtheorem{notation}[thm]{Notation}

\theoremstyle{remark}
\newtheorem{remark}[thm]{Remark}

\newcommand{\cF}{\mathcal{F}}
\newcommand{\cU}{\mathcal{U}}

\newcommand{\Z}{\mathbb{Z}}

\newcommand{\Q}{\mathbb{Q}}
\newcommand{\F}{\mathbb{F}}
\newcommand{\define}{\mathrel{\mathop:}=}

\newcommand{\cO}{\mathcal O}
\newcommand{\cP}{\mathcal P}

\newcommand{\id}{\mathrm{1}} 
\newcommand{\App}{\mathcal{A}} 
\newcommand{\sW}{W_0} 
\newcommand{\aW}{{W}} 
\newcommand{\sS}{S_0} 
\newcommand{\aS}{{S}} 
\newcommand{\Cf}{\mathcal{{C}}_{f}} 
\newcommand{\fa}{{\bf{a}}} 
\newcommand{\Cw}{\mathcal{{C}}} 
\newcommand{\aH}{\mathcal{H}} 

\newcommand{\Labl}{{\mathcal{L}}}
\newcommand{\LablF}{\Labl\cF} 

\newcommand{\G}{\Gamma} 
\newcommand{\LG}{\operatorname{L}\!\Gamma}

\newcommand{\type}{\mathrm{type}} 

\newcommand{\x}{\mathbf{x}}
\newcommand{\y}{\mathbf{y}}
\newcommand{\z}{\mathbf{z}}

\newcommand{\kk}{\mathbbm{k}}

\newcommand{\SL}{\operatorname{SL}}

\newcommand{\Sh}{\operatorname{Sh}} 
\renewcommand{\star}{\operatorname{star}} 
\newcommand{\gate}{\operatorname{gate}} 

\numberwithin{equation}{subsection}

\definecolor{amethyst}{rgb}{0.6, 0.4, 0.8}
\definecolor{kellygreen}{rgb}{0.3, 0.73, 0.09}
\definecolor{americanrose}{rgb}{1.0, 0.01, 0.24}

\begin{document}

\hypersetup{pdfauthor={Milicevic, Schwer, Thomas},pdftitle={Put title here}}

\title[Chimney retractions encode orbits in affine flags]{Chimney retractions in affine buildings encode orbits in affine flag varieties}

\author{Elizabeth Mili\'{c}evi\'{c}}
\address{Elizabeth Mili\'{c}evi\'{c}, Department of Mathematics \& Statistics, Haverford College, 370 Lancaster Avenue, Haverford, PA, USA}
\email{emilicevic@haverford.edu}

\author{Petra Schwer}
\address{Petra Schwer, Department of Mathematics, Universitätsplatz 2, Otto-von-Guericke University of Magdeburg, Germany}
\email{petra.schwer@ovgu.de}

\author{Anne Thomas}
\address{Anne Thomas, School of Mathematics \& Statistics, Carslaw Building F07,  University of Sydney NSW 2006, Australia}
\email{anne.thomas@sydney.edu.au}

\thanks{The first author was partially supported by NSF Grant DMS \#2202017. The second author was partially supported by DFG Grant SCH 1550/5-1.  This research was partly supported by ARC Grant DP180102437.}


\begin{abstract}
This paper determines the relationship between the geometry of retractions and the combinatorics of folded galleries for arbitrary affine buildings, and so provides a unified framework to study orbits in affine flag varieties. We introduce the notion of labeled folded galleries for any affine building $X$ and use these to describe the preimages of chimney retractions.   When $X$ is the building for a group with an affine Tits system, such as the Bruhat--Tits building for a group over a local field, we can then relate labeled folded galleries and shadows to double coset intersections in affine flag varieties.  This result generalizes the authors' previous joint work with Naqvi on groups over function fields.
\end{abstract}

\maketitle

\begin{center}
	\emph{This paper is dedicated to the memory of Jacques Tits, for building \\ fundamental bridges between geometry and group theory.}
\end{center}


\bigskip


\section{Introduction}\label{sec:intro}

Jacques Tits introduced buildings as a unifying framework for studying algebraic groups over arbitrary fields.  In Tits' original construction of spherical buildings, a building is a simplicial complex associated to and acted upon by a semisimple algebraic group.  Bruhat and Tits later constructed affine buildings, for algebraic groups over local fields.  Many features of these spherical and affine buildings correspond to delicate algebraic structures in the associated group.  However, as established by Tits and others, not all buildings arise from a group.  In this paper, we use the rich geometry and combinatorics of arbitrary affine buildings to capture the relationships between retractions (mappings of the entire building onto a single apartment) and minimal galleries (shortest paths between alcoves in a given apartment).  We then use these relationships to give unified proofs of purely group-theoretic statements concerning flag varieties in all situations where there is an associated affine building.



\subsection{Chimney retractions and folded galleries in affine buildings}\label{sec:chimney-intro}

We work in the setting of an \emph{affine building} $X$, in which each apartment is a copy of the Coxeter complex for an irreducible affine Coxeter system of type $(\aW, \aS)$ and rank $n$. The maximal simplices of $X$ are \emph{alcoves}, and in any fixed apartment, the alcoves are in bijection with the elements of $\aW$. \emph{Retractions} from $X$ to a fixed apartment are simplicial maps which permit one to study the entire building by focusing on a single apartment.

All of the well-known examples of retractions in affine buildings can be simultaneously generalized as retractions from \emph{chimneys}. Introduced by Rousseau~\cite{Rousseau77,Rousseau01}, chimneys are related to the generalized sectors defined by Caprace and L\'ecureux \cite{CapraceLecureux}, who extended to arbitrary buildings work of Guivarc'h and Remy on Bruhat--Tits buildings~\cite{GuivarchRemy}. Since we focus on affine buildings, a chimney in this paper is parameterized by an element $y \in \aW$ together with a choice of a spherical standard parabolic subgroup of $\aW$, equivalently an indexing set $J \subseteq [n] = \{ 1,\dots,n\}$. Following the authors' joint work with Naqvi~\cite{MNST}, our $(J,y)$-\emph{sectors} are then a special case of Rousseau's chimneys.

\begin{figure}[h]
	\begin{minipage}[b]{0.45\linewidth}
		\centering
		\includegraphics[width=0.9\textwidth]{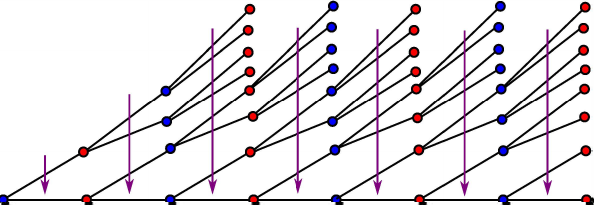}
	\end{minipage}
	\hspace{0.5cm}
	\begin{minipage}[b]{0.45\linewidth}
		\centering
		\includegraphics[width=0.9\textwidth]{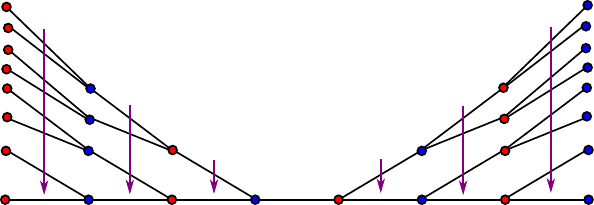}
	\end{minipage}
	\caption{Retractions of a tree from a chamber at infinity (left), and from an alcove (right).}
	\label{fig:treeRetractionsIntro}
\end{figure}

A \emph{chimney retraction} folds the entire building $X$ away from the $(J,y)$-chimney onto a single apartment.  In the extreme cases where $J$ is either empty or the whole set $[n]$, the chimney retraction coincides with the familiar retraction from either a chamber at infinity or the base alcove, respectively.  These are depicted in Figure \ref{fig:treeRetractionsIntro} for the case where $X$ is a tree.  
If $J$ is both nonempty and proper, in which case $X$ must be at least two-dimensional, the geometry of the associated chimney retractions interpolates between the geometry of these two classical retractions, as illustrated by Figure \ref{fig:chimneyRetraction}; see Examples \ref{ex:2Dfigure} and~\ref{ex:2Dfigure-retract} for a detailed discussion of this figure. In this more general situation, chimney retractions were first described by G\"ortz, Haines, Kottwitz, and Reuman \cite{GHKRadlvs}, and also appeared in a subsequent paper of Haines, Kapovich, and Millson \cite{HKM}.  In~\cite{MNST}, we develop a more precise framework for chimney retractions in the setting of an arbitrary affine building.  In \cite{MST2}, the authors apply this machinery to answer questions about certain affine Deligne--Lusztig varieties, which were the original motivation for the discussion of chimney retractions in \cite{GHKRadlvs}.

  \begin{figure}[h]
		\resizebox{0.75\textwidth}{!}
		{\begin{overpic}{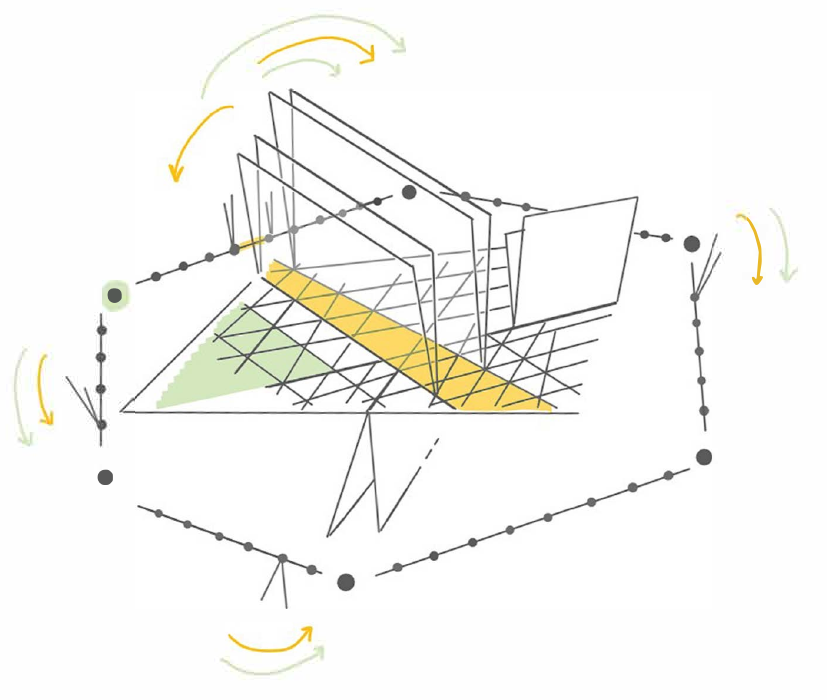}
			\end{overpic}}
		\caption{Two types of chimney retractions in a two-dimensional affine building.}	
		\label{fig:chimneyRetraction}	
\end{figure}

Any chimney in an affine building induces an \emph{orientation} on the apartments in $X$, as defined by Graeber and the second author \cite{GraeberSchwer}, such that each crossing and fold in any combinatorial \emph{gallery} inherits an assignment as positive or negative.  \emph{Positively folded} galleries were introduced by Gaussent and Littelmann \cite{GaussentLittelmann} and generalized by Ram \cite{Ram}, with both~\cite{GaussentLittelmann} and~\cite{Ram} working in the context of representations of complex semisimple algebraic groups. \emph{Shadows}, which were introduced in \cite{GraeberSchwer} and generalized in \cite{MNST}, provide a convenient combinatorics for studying a related collection of galleries which are positively folded with respect to a given \emph{chimney orientation}.  Figure~\ref{fig:shadows}, which is discussed in Example \ref{ex:A2shadow}, illustrates a family of shadows; see also the survey on shadows by the second author \cite{SchwerShadows}.

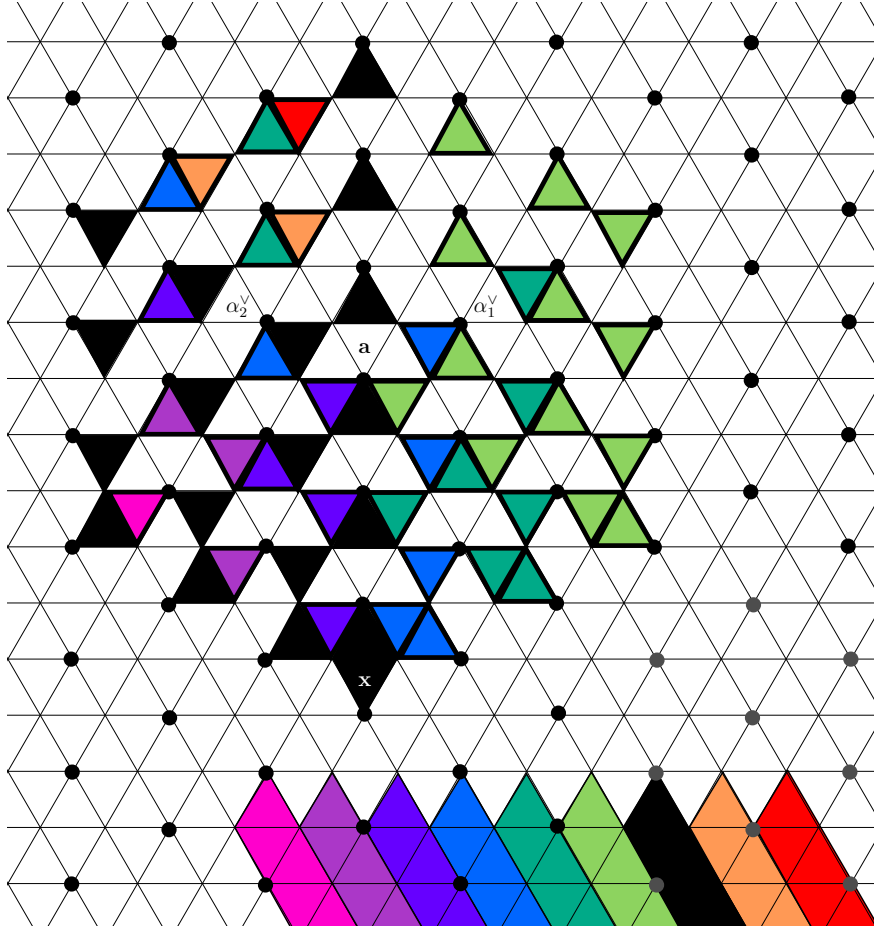
\begin{figure}[ht]
	\begin{center}
		\resizebox{0.75\textwidth}{!}
		{
			\begin{overpic}{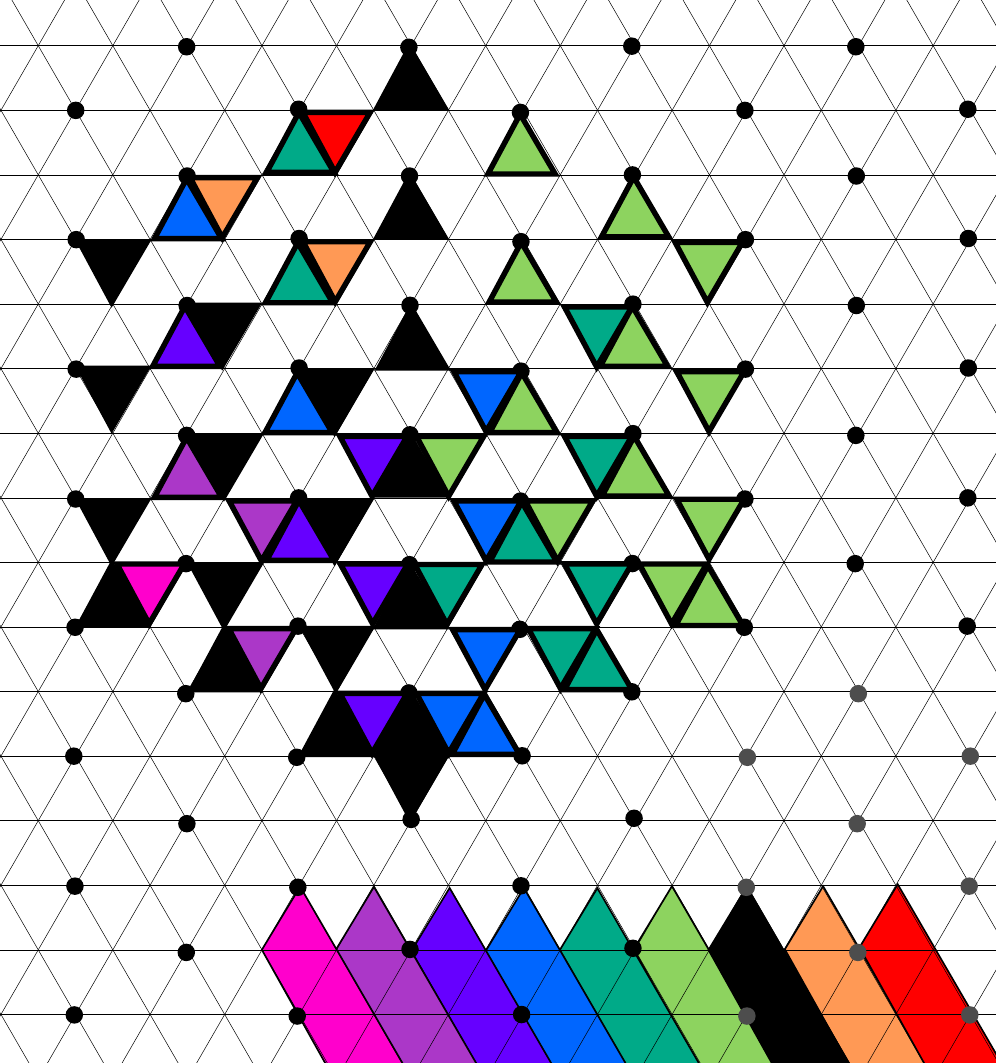}
			\put(38,62){$\fa$}
			\put(38,26){\color{white}{$\x$}}
			\put(50.5,66.5){$\alpha_1^\vee$}
			\put(23.7,66.5){$\alpha_2^\vee$}
			\end{overpic}
		}
		\caption{Some shadows in type $\tilde{A}_2$, for chimneys with $J = \{ 1 \}$ and various $y \in \aW$; e.g.~the black alcoves form the shadow of $\x = t^{-3(\alpha_1^\vee + \alpha_2^\vee)} \fa$ with respect to the $(J,\id)$-chimney represented by the black sector.}
		\label{fig:shadows}
	\end{center}
\end{figure}

\emph{Labeled} folded galleries were introduced by Parkinson, Ram, and Schwer (using the terminology ``labeled folded alcove walks") in order to describe points in certain orbits in the affine flag variety \cite{PRS}.  The authors employ this combinatorics to study families of affine Deligne--Lusztig varieties in \cite{MST1}.  In \cite{MNST}, the authors generalize the labeled folded galleries of \cite{PRS} to the setting of being positively folded with respect to any chimney orientation, rather than the special case which induces a periodic orientation. However, the mechanics of the labeling as defined in \cite{PRS,MNST} are only valid for those Bruhat--Tits buildings which arise from groups over function fields, as opposed to arbitrary local fields.  The primary goal of this paper is to develop labeled folded galleries, and purely geometric methods for working with them, in any affine building, and so fully generalize the main algebraic results of \cite{MNST}.


\subsection{Main results}\label{sec:theorems}

In this section, we state three of the main theorems contained in this paper. The first result places preimages of chimney retractions in any affine building in bijection with labeled folded galleries.  The second two theorems then apply this geometric framework to the algebraic context of orbits in affine flag varieties and affine Grassmannians, which are captured as shadows of chimney retractions.

As above, let $X$ be any affine building, of type $(\aW,\aS)$ irreducible affine of rank $n$.  We will assume $X$ is thick and panel-regular, as defined in Section~\ref{sec:buildings-prelim} (these assumptions hold for all Bruhat--Tits buildings, for instance).  Fix a standard apartment $\App$ in $X$ and base alcove $\fa$ in $\App$. Denote the alcove in $\App$ corresponding to the element $x \in W$ using  boldface $\x = x\fa$. Any minimal gallery from the base alcove $\fa$ to $\x$ corresponds to a reduced expression $x = s_{j_1}\cdots s_{j_l}$, where $s_{j_i} \in S$, and such a gallery has type $\vec{x} = (j_1, \dots, j_l)$. Given any $J \subseteq [n]$ and $y \in \aW$, denote by $r_{J,y}: X \to \App$ the chimney retraction onto $\App$ which folds all simplices in $X$ away from the $(J,y)$-chimney. In the special case $J = [n]$ and $y=1$, the chimney retraction is centered at the base alcove, and we write $r: X\to \App$. A gallery in $\App$ is positively folded if all folds face away from the $(J,y)$-chimney.  The shadow $\Sh_{J,y}(\x)$ is the set of final alcoves of all galleries of type $\vec{x}$ which start at $\fa$ and are positively folded with respect to the $(J,y)$-chimney. 

In this setting, we introduce the notion of a $(J,y)$-labeled folded gallery.  This consists of a gallery which is positively folded with respect to the $(J,y)$-chimney, together with ``labels" on certain of its crossings and folds; see Section~\ref{sec:LFGs} for the precise definition.  In Remark~\ref{rem:comparisonLF}, we discuss how our formulation relates to that of the labeled folded galleries already in the literature.

The main geometric theorem of this paper is then as follows.

\begin{restatable}{thm}{chimneyLFGsRestate} 
\label{thm:chimneyLFGs} Let $X$ be a thick, panel-regular affine building, of type $(W,S)$ irreducible affine of rank $n$.
Let $J \subseteq [n]$ and $x,y,z \in \aW$.  Then 
 the set \[r^{-1}(\x) \cap r_{J,y}^{-1}(\z)\] is in bijection with the set of $(J,y)$-labeled folded galleries which are of type $\vec{x}$ and have final alcove $\z$.
\end{restatable}

Now let $G$ be any group which admits an affine building $X$ of type $(\aW,\aS)$. Let~$I$ denote the stabilizer of $\fa$ in $G$, often called the Iwahori subgroup. The group $G$ admits a disjoint decomposition into double cosets of the form $IxI$ where $x \in \aW$. Let $P$ be a standard spherical parabolic subgroup of $G$, which corresponds to a subset $J \subseteq [n]$.  
The group $P=P_J$ admits a Levi decomposition $P = LU$, where $L$ is the Levi component and $U$ its unipotent radical. For any $P$, define a subgroup $I_P$ of $G$ by  $I_P = (I \cap L)U$. 
Given any $y \in \aW$, we denote by $(I_P)^y$ the conjugate $y I_P y^{-1}$.

Our next result simultaneously generalizes Theorems 1.2 and 5.11 from \cite{MNST}, by moving beyond the case of function fields, as well as recording the statement in terms of shadows.

\begin{thm}\label{thm:DoubleCosetsI-intro}  
Let $x, y, z \in \aW$, and let $P=P_J$ be a standard spherical parabolic subgroup of $G$.  Then there is a bijection between the points of the intersection 
	\[  I x I  \cap (I_P)^y z I \] 
	and the set of $(J,y)$-labeled folded galleries which are of type $\vec{x}$ and have final alcove $\z$. Moreover, 
	 \[  I x I  \cap (I_P)^y z I  \neq \emptyset \] if and only if the alcove $\z$ lies in the shadow $\Sh_{J,y}(\x)$. 
\end{thm}

\noindent For example, in Figure~\ref{fig:shadows}, where  $\x = t^{-3(\alpha_1^\vee + \alpha_2^\vee)}\fa$ is the black alcove labeled in white, the intersection $IxI \cap I_P z I$ is nonempty if and only if $\z$ is one of the black alcoves.

We also extend Theorem \ref{thm:DoubleCosetsI-intro} to arbitrary parahoric subgroups of $G$; see Theorem~\ref{thm:DoubleCosetsPar}, which generalizes Theorems 5.14 and 5.16 from \cite{MNST}. We highlight an important special case of this extension as Theorem \ref{thm:DoubleCosetsK-intro} below, for which we require some additional terminology.

For any face $\sigma$ of the base alcove $\fa$ which contains the origin $v_0$, denote by $K(\sigma)$ the parahoric subgroup of $G$ which is the stabilizer in $G$ of $\sigma$.  For example, $I = K(\fa)$ in Theorem \ref{thm:DoubleCosetsI-intro}.  Write $K = K(v_0)$.  Denote by $t^\lambda$ the translation by the element~$\lambda$ of the coroot lattice $R^\vee$, where $\lambda$ is dominant if it lies in the Weyl chamber containing $\fa$.  For any $\lambda \in R^\vee$, denote by $\star(\lambda)$ the set of alcoves of $\App$ which contain the vertex $\lambda$.  Given any $J \subseteq [n]$ and $y \in \aW$, the shadow $\Sh_{J,y}(\lambda)$ of the vertex $\lambda \in R^\vee$ is the set of final vertices of all galleries of type $\vec{\lambda}$ which start at $v_0$ and are positively folded with respect to the $(J,y)$-chimney.

 The following result generalizes both Theorem 1.3 and Corollary 5.18 from \cite{MNST}. See also  Theorem \ref{thm:DoubleCosetsPar}, which is a generalization of both Theorems \ref{thm:DoubleCosetsI-intro} and \ref{thm:DoubleCosetsK-intro} where $I$ and $K$, respectively, are replaced by arbitrary parahoric subgroups of $G$ which stabilize faces of the base alcove containing the origin.
	
\begin{thm}\label{thm:DoubleCosetsK-intro}
Let $P=P_J$ be a standard spherical parabolic subgroup of $G$. Let $\lambda$ and~$\mu$ be in the coroot lattice with $\lambda$ dominant, and let $x_\lambda\fa$ be the alcove in $\star(\lambda)$ closest to $\fa$. For any $y \in \aW$, there is a bijection between the points of the intersection 
	\[Kt^\lambda K \cap (I_P)^y t^\mu K\] 
and the union over $w \in \sW$ of the set of $(J,y)$-labeled folded galleries which are of type $\overrightarrow{wx_\lambda}$ and have final alcove in $\star(\mu)$. Moreover, 
	\[Kt^\lambda K \cap (I_P)^y t^\mu K \neq \emptyset\]
	if and only if the vertex $\mu$ lies in the shadow $\Sh_{J,y}(\lambda)$.
\end{thm}

\noindent We refer the reader to Figures 1, 6, and 7 in \cite{MNST} for illustrations of Theorem~\ref{thm:DoubleCosetsK-intro} in types $\tilde{A}_2$ and $\tilde{C}_2$.  See  the discussion in Section 1.3 of \cite{MNST} for a history of results from the literature which are directly related to Theorems \ref{thm:DoubleCosetsI-intro} and/or \ref{thm:DoubleCosetsK-intro}.


\subsection{Discussion of related and further work}\label{sec:related}

We now discuss some specific open problems and applications which are related to the results of this paper.  

For $\lambda$ any coroot, the shadow $\Sh_{\emptyset,1}(\lambda)$ occurring in a special case of the statement of Theorem \ref{thm:DoubleCosetsK-intro}  coincides with the coroot lattice points which lie in the $\lambda$-Weyl polytope, as shown by the second author for all affine buildings \cite{Hitzel}.  Beyond this extreme case, however, far less is known about the geometry of shadows.  Recursive descriptions are obtained in \cite{GraeberSchwer, MNST}, but closed formulas appear somewhat elusive, as more general shadows are no longer convex, frequently giving rise to what we refer to as ``holes" and ``notches" in~\cite{MNST}.  As Figure \ref{fig:shadows} suggests (see also the many figures in \cite{MNST}), one might hope that chimney  shadows could be expressed using unions of certain convex regions depending on the chimney, perhaps even using unions of pseudo-Weyl polytopes in the sense of \cite{AndersonKogan, Kamnitzer}.

 Earlier work of the authors on affine Deligne--Lusztig varieties relies directly upon the algebraic statements in both \cite{PRS} and \cite{MNST}, relating labeled folded galleries to double coset intersections in affine flag varieties over function fields.  As such, these previous applications only address a certain family of arithmetic problems, even though the techniques in both \cite{MST1} and \cite{MST2} exclusively involve folded galleries in affine buildings. This paper thus provides the remaining ingredient for generalizing these earlier results to the context of any group admitting an affine Tits system. In particular, the nonemptiness statements and dimension formulas for affine Deligne--Lusztig varieties in \cite{MST1,MST2} now immediately extend from groups with values in the function field $\overline{\F}_q((t))$ to any other local field, such as the $p$-adic numbers $\Q_p$. 
 
 Some of our main theorems are algebraic in nature, and thus only directly apply to affine buildings associated to groups, such as Bruhat--Tits buildings.  However, the majority of  techniques make sense in a broader geometric context (e.g.~Theorem~\ref{thm:chimneyLFGs} and all material in Sections \ref{sec:preliminaries} through \ref{sec:retractions}), and it would be interesting to develop further applications to arbitrary buildings. For example, Abramenko, Parkinson, and Van Maldeghem work in any locally finite, regular building to count the number of points in intersections of spheres for the Weyl-distance \cite{APvM}, in a manner which can be translated into an appropriate analog of Theorem \ref{thm:DoubleCosetsPar}, and which has many applications as discussed in~\cite{APvM}.
 
Note that we provide only a sample of related questions in this section.  Since chimney retractions are so geometrically natural, we expect there to be many algebraic, combinatorial, and geometric  applications beyond those discussed here.


\subsection{Organization of the paper} 

We now provide a brief overview of the contents of the paper. All background material is contained in Sections \ref{sec:preliminaries} and \ref{sec:Chimneys}. Section~\ref{sec:preliminaries} includes all preliminary terminology and notation on affine Coxeter systems and affine buildings, largely following \cite{AB}.  Section \ref{sec:Chimneys} then reviews from \cite{MNST} the definitions of chimneys, retractions from chimneys, orientations induced by chimney retractions, and shadows of chimney retractions. 

All new material is contained in Sections \ref{sec:retractions} and \ref{sec:chimney-shadows}. We are able to generalize the corresponding theorem statements from \cite{MNST} to the setting of arbitrary affine buildings by combining results from \cite{MNST} for arbitrary affine buildings, which we restate in Section~\ref{sec:Chimneys}, with the material we develop in Section \ref{sec:retractions}, and results from~\cite{GHKRadlvs} which we formalize and generalize in Section~\ref{sec:chimney-shadows}.  More precisely, in Section \ref{sec:retractions} we define the notion of a labeled folded gallery in the context of any thick, panel-regular affine building, and then prove Theorem \ref{thm:chimneyLFGs}, which places preimages of chimney retractions in bijection with these upgraded labeled folded galleries.  In Section~\ref{sec:chimney-shadows}, we identify preimages of chimney retractions with certain double coset intersections, and so connect shadows of chimney retractions to intersections of certain orbits in any group which admits an affine Tits system.  The paper then concludes with the proofs of Theorems \ref{thm:DoubleCosetsI-intro} and \ref{thm:DoubleCosetsK-intro}, which are obtained by combining these earlier results.


\subsection*{Acknowledgements}

EM gratefully acknowledges the financial support of Haverford College for sponsoring a visit by AT in June 2022, during which part of this work was completed. The authors are indebted to the anonymous referee of \cite{MNST}, for an exceptionally careful reading which prompted the authors to appreciate that decisively new geometric techniques would be required to generalize the proofs of the algebraic statements from \cite{MNST} beyond the case of algebraic groups over function fields.


\section{Affine Coxeter systems and buildings}\label{sec:preliminaries}

This section reviews preliminary definitions and notation on affine Coxeter systems and buildings, assuming familiarity with references such as \cite{AB, Humphreys, Ronan}. 
We begin in Section \ref{sec:Coxeter-prelim} with some background and notation for affine Coxeter systems. Definitions of related geometric structures such as roots and alcoves are found in Sections \ref{sec:roots-prelim} and \ref{sec:alcoves-prelim}, respectively. To reduce notational complexity, we assume throughout the paper that the affine Coxeter system is irreducible.  All geometric results (for example, all material in Sections \ref{sec:preliminaries} through \ref{sec:retractions}) can be extended to reducible affine Coxeter systems, and we leave this natural extension as an exercise for the reader.

The language required for working with galleries in arbitrary affine buildings is then reviewed in Section \ref{sec:buildings-prelim}. We also relate this building-theoretic framework to the more specific context of groups with affine Tits systems, though this additional algebraic structure is not required until Section \ref{sec:chimney-shadows}, and so we postpone that discussion until later in the paper.


\subsection{Affine Coxeter systems}\label{sec:Coxeter-prelim}

Let $(\aW,\aS)$ be an irreducible affine Coxeter system of rank $n$, and let $V$ be the associated $n$-dimensional real vector space on which $\aW$ acts, which we can identify with $n$-dimensional Euclidean space.   If we denote the origin of $V$ by $v_0$, then $(\aW,\aS)$ has associated spherical Coxeter system $(\sW,\sS)$ such that $\sW$ is the stabilizer in $\aW$ of $v_0$ and $\sS = \{ s_1, \dots, s_n \}$ is the set of elements of $\aS = \{ s_0,s_1, \dots, s_n\}$ which fix $v_0$.  Denote the index set for $\sS$ by $[n]:= \{1,2,\dots, n\}$, and the index set for $\aS$ by $[\overline{n}] := {0} \cup [n]$.  We write $\ell:\aW \to \Z$ for the length function of the Coxeter system $(\aW,\aS)$, and denote the longest element of the spherical Coxeter group $\sW$ by $w_0$. We typically use the letters $x, y, z$ for elements of $\aW$ and $u, v, w$ for elements of $\sW$.  

The vector space $V$ admits an ordered basis which we denote by $\Delta = (\alpha_i)_{i \in [n]}$, and a symmetric bilinear form $B(\alpha_i, \alpha_j) = -\cos \frac{\pi}{m(i,j)}$, where $m(i,j)$ is the $(i,j)$-entry of the associated Coxeter matrix.  Given any $i \in [n]$, define $\alpha_i^\vee \in V^*$ by $\langle \alpha_i^\vee, v \rangle := 2B(\alpha_i, v)$ for any $v \in V$, where $\langle \cdot, \cdot \rangle: V^* \times V \to \Z$ is the evaluation pairing. The ordered set $\Delta^\vee = (\alpha_i^\vee)_{i \in [n]}$ is then a basis for $V^*$. The action of $\aW$ on $V$ induces an action of $\aW$ on $V^*$, and every generator $s_i \in \sS$ acts as a linear reflection fixing the hyperplane $H_{\alpha_i} = \{ x \in V^* \mid \langle x, \alpha_i \rangle = 0\}$. We sometimes also write $s_i = s_{\alpha_i}$ for $i \in [n]$.


\subsection{Roots and coroots}\label{sec:roots-prelim}

The vectors $\Phi = \{w\alpha_i \mid w \in \sW, \ i \in [n]\} \subset V$ are called \emph{roots}, and the subset of \emph{positive roots} is $\Phi^+ = \{ \alpha \in \Phi \mid \langle x, \alpha \rangle \geq 0\ \text{for all}\ x \in \Cf\}$. The elements of $\Delta$ are called the \emph{simple roots}, and the elements of $\Delta^\vee$ are the \emph{simple coroots}. There is a partial ordering on the set of roots, given by $\alpha \geq \beta$ if and only if  $\alpha - \beta$ is a nonnegative linear combination of simple roots.  There is a unique maximal element in $\Phi$ with respect to this partial ordering, called the \emph{highest root}, typically denoted by $\widetilde\alpha$.

Given any subset $J \subseteq [n]$, denote by $W_J$ the standard parabolic subgroup of $\sW$ generated by $S_J = \{ s_j \mid j \in J \}$, and recall that $(W_J, S_J)$ is a spherical Coxeter system.  (Note that in the spherical context, we omit the 0-subscript in $(\sW)_J$ to avoid double indices.) We will write $\Phi_J$ for the sub-root system of $\Phi$ associated to $W_J$, which has positive roots $\Phi_J^+ = \Phi^+ \cap \Phi_J$ and basis of simple roots $\Delta_J = (\alpha_j)_ { j \in J }$.  In particular, if $J = \emptyset$, then $W_J$ is trivial and $\Phi_J = \Delta_J= \emptyset$, while if $J = [n]$ then $W_J = \sW$, $\Phi_J = \Phi$, and $\Delta_J = \Delta$.
The set of simple coroots corresponding to the standard parabolic subgroup $W_J$ is denoted by $\Delta^\vee_J$.

There is a distinguished lattice in $V^*$ which is preserved by the actions of $\aW$ and $\sW$, called the coroot lattice $R^\vee = \bigoplus \Z \alpha_i^\vee$. The partial ordering on $\Phi$ extends naturally to $R^\vee$. Denote by $t^\lambda$ the translation in $V^*$ by the coroot $\lambda \in R^\vee$. The set of all such translations $\{ t^\lambda \mid \lambda \in R^\vee\}$ is a subgroup of the affine Coxeter group $\aW$.  For any $\lambda, \mu \in R^\vee$, we have $t^\lambda t^\mu = t^{\lambda + \mu} = t^{\mu + \lambda} = t^\mu t^\lambda$, and given any $w \in \sW$, we have $w t^\lambda w^{-1} = t^{w\lambda}$.  Any element $x \in \aW$ can be expressed uniquely as $x = t^\lambda w$ for some $\lambda \in R^\vee$ and $w \in \sW$, and moreover $\aW \cong R^\vee \rtimes \sW$.


\subsection{Chambers, alcoves, and half-spaces}\label{sec:alcoves-prelim}

The simplicial cone defined by $\Cf = \{ x \in V^* \mid \langle x, \alpha_i \rangle \geq 0\ \text{for all}\ i \in [n]\}$ is called the \emph{fundamental} or \emph{dominant chamber}.  The simplicial cones $w\Cf \subset V^*$ for $w \in \sW$ are the \emph{chambers}. As such, the set of chambers is in bijection with the elements of $\sW$, and for $w \in \sW$ we thus typically write $\Cw_w$ for the chamber $w\Cf$.  The chamber $\Cw_{w_0} = w_0 \Cf$  is called the \emph{opposite} or \emph{antidominant chamber}, and consists of all points $x \in V^*$ such that $\langle x, \alpha_i \rangle \leq 0$ for all $i \in [n]$. 

Given any root $\alpha \in \Phi$, we write $s_{\alpha} \in \sW$ for the linear reflection on $V^*$ fixing  the hyperplane $H_{\alpha} = \{ x \in V^* \mid \langle x, \alpha \rangle = 0\}$ pointwise.  All reflections in $\sW$ are of this form. The set of \emph{walls} is then denoted by $\aH_0 =\{ H_\alpha \mid \alpha \in \Phi\} = \{  wH_{\alpha_i} \mid w \in \sW, \ i \in [n] \} $, and consists of the collection of hyperplanes in $V^*$ fixed by some reflection $s_\alpha \in \sW$.   A chamber is then also the closure of a maximal connected component of $V^* \setminus \aH_0$.

For each root $\alpha \in \Phi$ and integer $k \in \Z$, we write $H_{\alpha,k}$ for the affine hyperplane or \emph{wall} of $V^*$ given by $H_{\alpha,k} = \{ x \in V^* \mid \langle x, \alpha \rangle = k \}$. Write $s_{\alpha,k}$ for the affine reflection fixing $H_{\alpha,k}$ pointwise.  Each $s_{\alpha,k}$ is an element of $\aW$, and every reflection in $\aW$ is of this form.  Note that $H_{\alpha} = H_{\alpha,0}$ and $s_{\alpha} = s_{\alpha,0}$ for any $\alpha \in \Phi$. In addition, the affine Coxeter generator $s_0$ satisfies $s_0 = s_{\widetilde \alpha,1}$.

Denote the set of  all affine hyperplanes by $\aH = \{ H_{\alpha,k} \mid \alpha \in \Phi, k \in \Z\}$. The closure of a maximal connected component of $V^* \setminus \aH$ is called an \emph{alcove}.  Since $(\aW,\aS)$ is irreducible, each alcove is a simplex.  We write $\fa$ for the alcove bounded by the hyperplanes $\{ H_{\alpha_1},\dots,H_{\alpha_n}, H_{\widetilde \alpha,1} \}$, and call $\fa$ the \emph{base} or \emph{fundamental alcove}. In particular, $\Cf$ is the unique chamber containing $\fa$. The set of alcoves in $V^*$ is in bijection with the elements of $\aW$, and for $x \in \aW$ we write $\x$ for the alcove $x\fa$.   
A \emph{panel} is a codimension one face of an alcove (recalling that each alcove is a simplex), and the \emph{supporting hyperplane} of a panel $p$ is the unique element of $\aH$ which contains $p$.  A panel $p$ has \emph{type} $i \in [\overline{n}]$ if $p$ is the (unique) panel to be contained in two alcoves of the form $x\fa$ and $x s_i \fa$, where $x \in \aW$.  A face $\sigma$ of an alcove then has \emph{type} $J \subsetneq [\overline{n}]$ if $\sigma$ is the intersection of panels of types $j \in J$.  Note that the action of $\aW$ is type-preserving.

For any root $\alpha\in\Phi$, integer $k \in \Z$, and element $w\in\sW$, we denote by $\alpha^{k,w}$ the closed half-space of $V^*$ bounded by the hyperplane $H_{\alpha, k}$ that contains a subsector of the chamber $\Cw_w$.  In particular, for any $k \in \Z$ the half-space $\alpha^{k,\id}$ contains a subsector of the dominant chamber $\Cf$, and $\alpha^{k,w_0}$ contains a subsector of the antidominant chamber $\Cw_{w_0}$.   The group $\aW$ acts on the set of closed half-spaces $\{ \alpha^{k,w} \mid w \in \sW, k \in \Z \}$ via \[(t^\lambda v) \cdot \alpha^{k,w} = (v\alpha)^{k+\langle \lambda, v\alpha \rangle, vw},\] where $\lambda \in R^\vee$ and $v \in \sW$.


\subsection{Affine buildings and galleries}\label{sec:buildings-prelim}

Now let $X$ be an affine building of type $(\aW,\aS)$ irreducible. We regard $X$ as a simplicial complex, and for any apartment $\App$ of $X$ we may fix an identification of $\App$ with $V^*$ such that we may discuss chambers, roots, hyperplanes, alcoves, panels, and so on in the apartment $\App$.  We  refer to the closed half-spaces of $\App$ determined by hyperplanes in $\aH$ as \emph{half-apartments}. Under this identification, the Coxeter groups $\aW$ and $\sW$ admit natural actions on $\App$, as well as on (certain collections of) half-apartments.

Given a panel $p$ of $X$, 
the set of alcoves in $X$ which contain $p$ is denoted $\star(p)$.
The building $X$ is \emph{thick} if every panel of $X$ is contained in at least three distinct alcoves; that is, if $\star(p)$ has cardinality $\geq 3$ for every panel $p$.  The building $X$ is \emph{panel-regular} if for every $i \in [\overline{n}]$, there is a cardinality $N_i$ such that for every panel $p$ of type $i$, the set of alcoves $\star(p)$ has $N_i+1$ elements.  We do not assume that this cardinality $N_i$ is finite or countable.

We continue by reviewing some terminology related to galleries in affine buildings.  
The galleries in Definition \ref{def:CombGallery} below are special cases of the ones considered in \cite{GaussentLittelmann}.

\begin{definition}\label{def:CombGallery}
	A \emph{(combinatorial) gallery} is a sequence of alcoves $c_j$ and faces $p_j$ in the affine building $X$, say
		$\gamma=(p_0, c_0, p_1, c_1, p_2, \dots , p_l, c_l, p_{l+1} ),$
	where the first and last faces $p_0 \subseteq c_0$ and $p_{l+1}\subseteq c_l$ are simplices, and for $1 \leq j \leq l$ the face~$p_j$ is  a panel of both alcoves $c_{j-1}$ and $c_{j}$. If $p_0 = c_0$ and $p_{l+1} = c_l$, then we simplify notation by writing $\gamma = (c_0, p_1, c_1, p_2,\dots, p_l, c_l)$.
\end{definition}

\noindent The \emph{length} of the gallery $\gamma$ equals the index $l$ of its last alcove $c_l$;  equivalently, $\gamma$ has length one less than the number of alcoves it contains. A gallery $\gamma$ is \emph{minimal} if it has minimal length among all galleries with first face $p_0$ and last face $p_{l+1}$. 

 The length function on galleries yields a natural distance function on the set of alcoves of $X$, as follows.  Given alcoves $c$ and $c'$, the distance between $c$ and $c'$ is defined to be the length of some (hence any) minimal gallery with first alcove $c$ and last alcove $c'$.  One can then show that for any panel $p$ of $X$, and any alcove $d$ of $X$, there exists a unique alcove in $\star(p)$ which is at minimal distance from $d$ in this metric; see Section 5.3 in \cite{AB} for proofs.  This observation permits the following definition.

\begin{definition}\label{defn:gate}
	Let $p$ be a panel of $X$ and $d$ an alcove of $X$. The \emph{gate of $p$ with respect to $d$}, denoted by $\gate_d(p)$,  is the unique alcove in $\star(p)$ that is closest to $d$.  If $d$ is the fundamental alcove $\fa$, we simplify notation and write $\gate(p) = \gate_\fa(p)$.  
\end{definition} 

\noindent The gate of $p$ with respect to $d$ is sometimes also called the \emph{projection of $d$ onto $p$}.

Every gallery in $X$ is also equipped with a type indexed by elements of $\aS$, as follows.

\begin{definition}\label{def:type} Given a gallery $\gamma=(p_0, c_0, p_1, c_1, \dots , p_l, c_l, p_{l+1} )$, the \emph{type} of $\gamma$ is 
	\[
	\type(\gamma) := (\type(p_0), s_{i_1}, s_{i_2}, \dots, s_{i_l}, \type(p_{l+1})),
	\]
	where the panel $p_j$ has type ${i_j} \in [\overline{n}]$ for $1 \leq j \leq l$.  If $p_0 = c_1$ and $p_{l+1} = c_l$ then we write $\type(\gamma) = (s_{i_1}, s_{i_2}, \dots, s_{i_l})$. 
\end{definition}

\noindent For $x \in \aW$ and a fixed minimal gallery from $\fa$ to $\x$, we refer to any gallery of this same type with first face $\fa$ as having \emph{type~$\vec{x}$}. Similarly, for $\lambda \in R^\vee$ and a fixed minimal gallery from $v_0$ to $\lambda$, we refer to any gallery of this same type with first face $v_0$ as having \emph{type~$\vec{\lambda}$}.


\section{Chimney retractions and positively folded galleries}\label{sec:Chimneys}

This section provides a self-contained review of chimneys, as well as their induced orientations on and retractions of affine buildings, following~\cite{MNST}. Section \ref{sec:chimneys}  includes the definition of a chimney. Retractions from chimneys are reviewed in Section~\ref{sec:chimney-retractions}. The natural orientations on apartments and galleries induced by chimneys are discussed in Section \ref{sec:chimneyOrient}. In Section \ref{sec:shadows} we recall the definition of shadows, together with Proposition \ref{prop:shadowsRetractions} which interprets shadows in terms of certain chimney retractions.


\subsection{Chimneys and sectors}\label{sec:chimneys}

We first review the definition of chimneys and sectors in the standard apartment $\App$ of an affine building $X$. For a discussion of related literature and further examples, we refer the reader to \cite[Sec.~3.1]{MNST}.

\smallskip

\begin{definition}\label{def:J-chimney}
Let $J \subseteq [n]$, with corresponding root system $\Phi_J$.  The \emph{$J$-chimney} is the following collection of half-apartments of $\App$:
\[
\xi_J\define 
 \left\{\alpha^{k,\id} \mid \alpha \in \Phi_J^+, k\leq 0 \right\} 
\cup
\left\{\alpha^{k, w_0} \mid \alpha \in \Phi_J^+, k\geq 1\right\}
\cup
\left\{\beta^{k, w_0} \mid \beta\in \Phi^+ \setminus \Phi_J^+, k\in \Z\right\}.
\] 
For any collection of integers $\{n_\beta\in\Z \mid \beta\in \Phi^+ \setminus \Phi_J^+\}$, the corresponding \emph{$J$-sector} is the subcomplex of $\App$ given by 
\[
S_J\left(\{n_\beta\}_{\beta\in \Phi^+ \setminus \Phi_J^+}\right)
\define 
\left( \bigcap_{\alpha\in\Phi_J^+} \alpha^{0, \id} \cap\alpha^{1, w_0} \right) \cap
\left( \bigcap_{\beta\in \Phi^+ \setminus \Phi_J^+} \beta^{n_\beta, w_0}\right).
\]  
\end{definition}

Next we provide several examples of chimneys and sectors.

\begin{example}\label{ex:2Dfigure}
Figure~\ref{fig:chimneyRetraction} shows a piece of an affine building of type $\tilde A_2$, together with its wall trees at infinity. 
The green shaded chamber pointing left in the horizontal base apartment is an example of a $J$-sector with $J=\emptyset$.   A $J$-sector with $\vert J\vert =1$ is pictured in yellow; such a sector lies in between two adjacent hyperplanes for a single root direction and is periodic in the other root direction. In case $J = \{1,2\} = [2]$, then all $J$-sectors are single alcoves, such as any triangle depicted in the base apartment.
\end{example}

\begin{example}\label{ex:extremes}
Two well-studied families of chimneys in arbitrary rank are obtained from considering the extreme cases for $J \subseteq [n]$.
\begin{enumerate}
\item If $J = \emptyset$ so that $\Phi_J =\emptyset$, then the chimney $\xi_{J}$ is the collection of all half-apartments containing subsectors of the antidominant chamber $\Cw_{w_0}$, which can be identified with the chamber at infinity represented by $\Cw_{w_0}$. The corresponding $J$-sectors are all of the translates of $\Cw_{w_0}$.

\item If $J = [n]$ so that $\Phi_J=\Phi$, then the chimney $\xi_J$ is the collection of all half-apartments containing the base alcove $\fa$, and the only $J$-sector is the base alcove $\fa$ itself. 
\end{enumerate}
\end{example}	

We now extend the natural action of $\aW$ on the collection of half-apartments in $\App$ to $J$-chimneys and $J$-sectors; we refer the reader to \cite[Def.~3.7]{MNST} for explicit formulas in terms of half-apartments of $\App$.

\begin{definition}\label{def:Jy-chimney}
For any $J \subseteq [n]$ and any $y \in \aW$, the \emph{$(J,y)$-chimney} $\xi_{J,y}$ is obtained by acting on the $J$-chimney on the left by $y$; that is, $\xi_{J,y}\define y \cdot \xi_J.$
Similarly, for any collection $\{n_\beta\in\Z \mid \beta\in \Phi^+ \setminus \Phi_J^+\}$, the corresponding \emph{$(J,y)$-sector} is 
$ S_{J,y}(\{n_\beta\})\define y \cdot S_{J}(\{n_\beta\})$.
\end{definition}

\begin{example}\label{ex:Jy-extremes}
The following examples illustrate the action of $\aW$ in the special cases of $J$-chimneys from Example~\ref{ex:extremes}.
\begin{enumerate}
\item If $J = \emptyset$ and $y=t^\lambda w \in \aW$, then any $(J,y)$-sector is a translate of the chamber $w\Cw_{w_0} = \Cw_{ww_0}$, and the $(J,y)$-chimney can be identified with the chamber at infinity represented by the chamber $\Cw_{ww_0}$. 

\item If $J = [n]$ and $y \in \aW$, then the $(J,y)$-chimney is the collection of all half-apartments containing the alcove $\y = y\fa$, and the only $(J,y)$-sector is the alcove $\y$ itself.
\end{enumerate}
\end{example}


\subsection{Chimney retractions}\label{sec:chimney-retractions}

We now review the definition of a retraction from a chimney; see \cite[Sec.~3.2]{MNST} for more details, including a discussion of related literature.

\begin{prop}[Proposition 3.15 of \cite{MNST}]\label{prop:J-apartment}
Let $J \subseteq [n]$ and $y \in \aW$.  Then for every alcove $c$ in $X$ there exists a collection of integers $\{n_\beta\in\Z \mid \beta\in \Phi^+ \setminus\Phi_J^+\}$  and an apartment $\App_{c,(J,y)}$ in the complete apartment system of $X$ such that $\App_{c,(J,y)}$ contains both $c$ and the $(J,y)$-sector $S_{J,y}(\{n_\beta\})$.  
\end{prop}

\begin{definition}\label{def:chimney-retraction}
Let $J \subseteq [n]$ and $y \in \aW$.   
For any alcove $c$ of $X$, let $r_{J,y}(c)$ be the image of $c$ under the unique isomorphism that maps an apartment $\App_{c,(J,y)}$ as in the statement of Proposition~\ref{prop:J-apartment} onto $\App$, while fixing $\App_{c,(J,y)} \cap \App$ pointwise.  The resulting induced simplicial map  $$r_{J,y}:X\to \App$$ is the \emph{retraction from the $(J,y)$-chimney}.    If $y = \id$, we call $r_J = r_{J,\id}$ the \emph{retraction from the $J$-chimney}. A \emph{chimney retraction} is a retraction from the $(J,y)$-chimney for some $J \subseteq [n]$ and $y \in \aW$.
\end{definition}

Note that any chimney retraction $r_{J,y}$ can also be applied to a gallery $\gamma$ in $X$ by applying the retraction $r_{J,y}$ simultaneously to all simplices in $\gamma$.  We will make repeated use, without comment, of the fact that chimney retractions are type-preserving, and hence in particular $r_{J,y}(\gamma)$ is a gallery of the same type as $\gamma$.

\begin{figure}[ht]
	\begin{minipage}[b]{0.45\linewidth}
		\centering
		\includegraphics[width=0.9\textwidth]{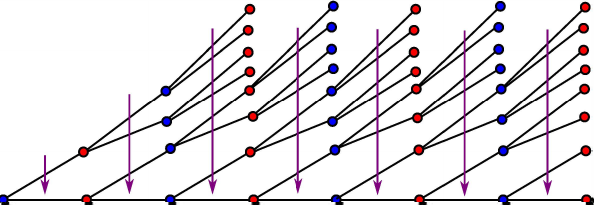}
	\end{minipage}
	\hspace{0.5cm}
	\begin{minipage}[b]{0.45\linewidth}
		\centering
		\includegraphics[width=0.9\textwidth]{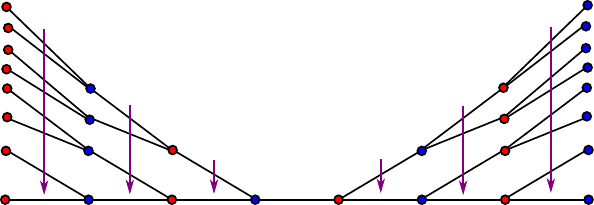}
	\end{minipage}
	\caption{Retractions of a tree from a chamber at infinity (left), and from an alcove (right).}
	\label{fig:treeRetractions}
\end{figure}

 \begin{figure}[h]
		\resizebox{0.75\textwidth}{!}
		{\begin{overpic}{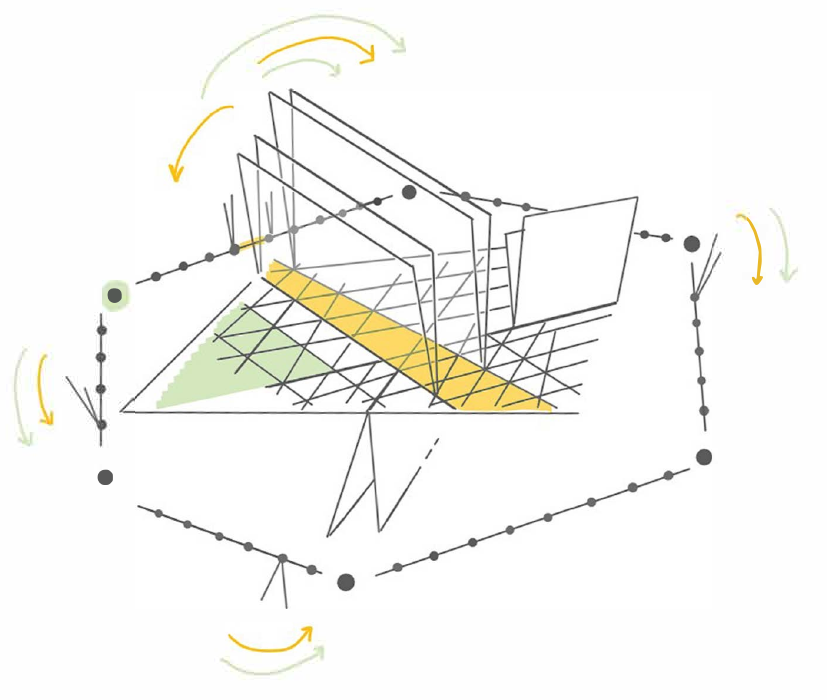}
			\end{overpic}}
		\caption{Chimney retractions in a two-dimensional affine building, where the chimney is a represented by a chamber (in green) and is represented by a $J$-sector with $|J|$=1 (in yellow).}	
		\label{fig:chimneyRetractionExample}	
\end{figure}

\begin{example}\label{ex:2Dfigure-retract}
Examples of chimney retractions are shown in Figures~\ref{fig:treeRetractions} and~\ref{fig:chimneyRetractionExample}.

In the tree case illustrated by Figure \ref{fig:treeRetractions}, the only possible chimney retractions are the classical retractions based at chambers at infinity or alcoves. These two familiar retractions are illustrated in Figure~\ref{fig:treeRetractions}. On the left, we see the retraction based at a chamber at infinity, located to the left of the horizontal base apartment. On the right, we show the retraction based at an alcove, which here is the edge in the tree depicted at the center of the figure.

Figure~\ref{fig:chimneyRetractionExample} illustrates the situation in a two-dimensional affine building. A chamber, equivalently a $\emptyset$-sector, is drawn in green pointing left in the horizontal base apartment. The effect of the corresponding retraction on the branching apartments shown folds every half-apartment away from the chamber at infinity corresponding to this $\emptyset$-sector, as indicated by the green arrows.  Now compare the action of the yellow arrows, which illustrate the retraction induced by the yellow $J$-sector such that $|J|=1$.  The half-apartments branching from the two hyperplanes along the boundary of the yellow $J$-sector fold down onto the base apartment in two opposite directions, as indicated by the yellow arrows.  For half-apartments which branch from hyperplanes not parallel to the boundary of the yellow $J$-sector, these fold in the same direction as the green arrows; namely, away from the chamber at infinity corresponding to the green $\emptyset$-sector.

 Moreover, these two chimney retractions on the two-dimensional building from Figure \ref{fig:chimneyRetraction} induce the two different kinds of retractions depicted in Figure \ref{fig:treeRetractions} on the wall trees at infinity.  Looking at the wall tree on the upper left in Figure \ref{fig:chimneyRetractionExample} containing the shaded yellow alcove, the $J$-chimney retraction coincides with the alcove retraction on the tree, as shown on the right in Figure \ref{fig:treeRetractions}. The $\emptyset$-chimney retraction on this same wall tree in Figure \ref{fig:chimneyRetraction} coincides with the retraction from the chamber at infinity indicated by the green dot, resulting in the type of retraction seen on the left in Figure \ref{fig:treeRetractions}.

\end{example}

We now discuss the retractions from those $(J,y)$-chimneys considered in Example~\ref{ex:Jy-extremes}.

\begin{example}\label{ex:Jy-retraction}
Both of the following chimney retractions are well-known retractions of affine buildings in arbitrary rank.
\begin{enumerate}
\item If $J = \emptyset$ and $y=t^\lambda w \in \aW$, then $r_{J,y}$ is the retraction onto $\App$ centered at the chamber at infinity represented by the chamber $\Cw_{ww_0}$. 

\item If $J = [n]$ and $y \in \aW$, then $r_{J,y}$ is the retraction onto $\App$ centered at the alcove $\y = y\fa$. 
\end{enumerate}
\end{example}


\subsection{Chimney orientations and positively folded galleries}\label{sec:chimneyOrient}

We next review some terminology concerning orientations on apartments and positively folded galleries.  We refer the reader to \cite{GraeberSchwer} for more background on this material.

\begin{definition}\label{def:orientation} An \emph{orientation} of the apartment $\App$ is a map $\phi$ that assigns to every pair $(c,p)$ consisting of an alcove $c$ of $\App$ and a panel $p$ of $c$ one of the symbols $+$ or $-$. If $\phi(c,p)=+$ (resp.~$\phi(c,p)=-$), we say that the alcove $c$ is on the \emph{positive} (resp.~\emph{negative}) side of $p$.  
\end{definition}

Any orientation of the apartment $\App$ gives rise to orientations on crossings of galleries at its panels, as follows.

\begin{definition}\label{def:posCross} Let $\phi$ be an orientation of $\App$, and let $\gamma = (p_0,c_0,p_1,\dots, c_l,p_{l+1})$ be a gallery in $\App$.  For $1 \leq j \leq l$, if $c_{j-1}\neq c_j$ we say that \emph{$\gamma$ has a crossing at $p_j$}.  Moreover, if $\gamma$ has a crossing at $p_j$, then it has a:
\begin{enumerate} 
\item \emph{positive crossing at $p_j$ with respect to $\phi$} if $\phi(c_{j-1},p_j) = -$ and $\phi(c_j,p_j)=+$; or
\item \emph{negative crossing at $p_j$ with respect to $\phi$} if $\phi(c_{j-1},p_j) = +$ and $\phi(c_j,p_j)=-$.  \end{enumerate}	
\end{definition}

Galleries can also have alcoves which are repeated in succession, resulting in what we refer to as \emph{folds}, and these folds also inherit an orientation from that of $\App$. 

\begin{definition}\label{def:posFold} Let $\phi$ be an orientation of $\App$, and let $\gamma = (p_0,c_0,p_1,\dots, c_l,p_{l+1})$ be a gallery in $\App$.   For $1 \leq j \leq l$, if $c_{j-1} = c_j$ we say that $\gamma$ is \emph{folded at $p_j$.} Moreover, if $\gamma$ is folded at $p_j$, then it is:
\begin{enumerate}
\item \emph{positively folded at $p_j$ with respect to $\phi$} if $\phi(c_j, p_j)=+$; or 
\item \emph{negatively folded at $p_j$ with respect to $\phi$} if $\phi(c_j, p_j)=-$.
\end{enumerate}
 	We say that $\gamma$ is \emph{positively (resp.~negatively) folded with respect to  $\phi$} if for all $1 \leq j \leq l$ such that $\gamma$ is folded at $p_j$, we have $\phi(c_j, p_j)=+$ (resp.~$\phi(c_j, p_j)=-$).
\end{definition}

\noindent In both Definitions~\ref{def:posCross} and \ref{def:posFold}, if the orientation is clear we may omit the phrase ``with respect to $\phi$".  Any orientation on $\App$ can thus give rise to positively folded galleries; see \cite[Sec.~2.2]{MNST} for further discussion.  

Any $(J,y)$-chimney induces a natural orientation on $\App$ in the sense of Definition~\ref{def:orientation}, as we now review. We refer the reader to \cite[Sec.~3.3]{MNST} for more details.

\begin{definition}\label{def:chimneyOrientation} 
Let $J \subseteq [n]$ and $y \in \aW$.  Let $c$ be an alcove of $\App$, let $p$ be a panel of $c$, and let $H$ be the hyperplane of $\App$ supporting $p$.  The \emph{orientation induced by the $(J,y)$-chimney}, denoted $\phi_{J,y}$, is defined by $\phi_{J,y}(c,p) = +$ if and only if the half-apartment determined by $H$ which contains $c$ is \emph{not} an element of the $(J,y)$-chimney.  A gallery in $\App$ is \emph{positively folded with respect to the $(J,y)$-chimney} if it is positively folded with respect to the orientation $\phi_{J,y}$.
\end{definition}

Intuitively, an alcove $c$ is on the positive side of $p$ with respect to $\phi_{J,y}$ if $c$ is on the side of $p$ which ``faces away from'' the $(J,y)$-chimney.  Similarly, a gallery is positively folded with respect to a chimney if it is ``folded away from'' the chimney.

\begin{example}\label{ex:Jy-retraction}
This example continues our analysis of the two well-known special cases of chimneys, which recover familiar orientations from the literature.
\begin{enumerate}
\item If $J = \emptyset$ and $y=1$, then the orientation $\phi_{J,1}$ coincides the periodic orientation on hyperplanes considered in \cite{PRS}. For $J = \emptyset$ and $y=t^\lambda w \in \aW$, the orientation $\phi_{J,y}$ agrees with the orientation at infinity corresponding to $w \in \sW$ from \cite{MST1}.

\item If $J = [n]$ and $y \in \aW$, then $\phi_{J,y}$ coincides with the opposite orientation induced by the alcove $\y = y\fa$ from \cite{GraeberSchwer}.
\end{enumerate}
\end{example}


\subsection{Shadows and retractions }\label{sec:shadows}

This section reviews the definition of shadows following Section 3.4 of \cite{MNST}, which in turn generalizes the regular shadows original to \cite{GraeberSchwer}.

\begin{definition}\label{def:shadow}  
	Let $J \subseteq [n]$ and $y \in \aW$. Let $x \in \aW$, and let $\sigma$ and $\tau$ be faces of the fundamental alcove $\fa$ which contain the origin.  Fix a minimal gallery $\gamma$ with first face $\sigma$ and last face $x\tau$. 
	The \emph{shadow of $x\tau$ starting at $\sigma$ with respect to the $(J,y)$-chimney}, denoted $\Sh_{J,y}(x\tau, \sigma)$, is the set of final simplices of all galleries of $\type(\gamma)$ which have first face~$\sigma$ and are positively folded with respect to the $(J,y)$-chimney (equivalently, with respect to the orientation $\phi_{J,y}$).
\end{definition}

In the following example, we discuss Figure~\ref{fig:shadows} which illustrates a family of alcove shadows, meaning that $\sigma = \tau = \fa$.

\begin{example}\label{ex:A2shadow}
Figure~\ref{fig:shadows} depicts a family of shadows in type $\tilde{A}_2$ with $J = \{ 1 \} \subseteq [2]$.  In this figure, we see the shadow of the alcove $\x = t^{-3(\alpha_1^\vee + \alpha_2^\vee)} \fa$ with respect to the $(J,y)$-chimney for several different choices of $y \in \aW$. The shadow of $\x$ with respect to the $(J,\id)$-chimney is the collection of alcoves colored black (including $\x$ itself); a sector representing this chimney is also colored black, located toward the lower right corner of the figure.

The other sectors along the bottom of Figure \ref{fig:shadows} represent $(J,y)$-chimneys for various translations $y \in \aW$.  For each such $(J,y)$-sector, of color say $C$, the corresponding shadow of $\x$ is the union of the black alcoves, the alcoves of color $C$, and all alcoves which are the same color as the sectors strictly between the black sector and the sector of color $C$.
\end{example}

We refer the reader to the figures in \cite{MNST} for further illustrations of various families of shadows in types $\tilde A_2$ and $\tilde C_2$.

\begin{notation}\label{not:shadow}  We use simpler notation for the most common cases of shadows.
	\begin{enumerate}
		\item If  $\sigma$ and $\tau$ both equal the fundamental alcove $\fa$, so that $x\tau = x\fa = \x$, the shadow of $\x$ is the set of final alcoves of galleries of type $\vec{x}$ which have first face and first alcove $\fa$ and which are positively folded with respect to the $(J,y)$-chimney.   We denote this shadow by $\Sh_{J,y}(\x)$.  If further $J = [n]$, we denote the shadow by $\Sh_\y(\x)$.
		
		\item If $\sigma$ and $\tau$ both equal the origin, so that $\lambda:=x\tau \in R^\vee$, the shadow of $\lambda$ is the set of final vertices of galleries of type $\vec{\lambda}$ which have first face the origin and which are positively folded with respect to the $(J,y)$-chimney.  We denote this shadow by $\Sh_{J,y}(\lambda)$.  If further $J = [n]$, we denote this shadow by $\Sh_\y(\lambda)$.
	\end{enumerate}
\end{notation}

Recall that $r_{J,y}: X\to \App$ denotes the retraction onto $\App$ from the $(J,y)$-chimney. 
Let $r:X \to \App$ denote the classical retraction $r_{[n],1}$ centered at the base alcove~$\fa$.

\begin{prop}[Proposition 3.32 of~\cite{MNST}]\label{prop:shadowsRetractions}
	Let $J \subseteq [n]$ and $y \in \aW$.   
	\begin{enumerate}
		\item For all $x \in \aW$, the shadow of $\x = x\fa$ with respect to the $(J,y)$-chimney is the set of alcoves given by 	
		\[
		\Sh_{J,y}(\x) =  r_{J,y} \circ r^{-1}(\x).
		\]
		\item For all $\lambda \in R^\vee$, the shadow of $\lambda$ with respect to the $(J,y)$-chimney is the set of vertices in $R^\vee$ given by 
		\[
		\Sh_{J,y}(\lambda) = r_{J,y} \circ r^{-1}(\sW \cdot \lambda).
		\]
	\end{enumerate}
\end{prop}

The shadows occurring in Proposition~\ref{prop:shadowsRetractions} will be connected to certain double cosets of groups acting on $X$ later in  Section \ref{sec:chimney-shadows}.


\section{Retractions and labeled folded galleries in affine buildings}\label{sec:retractions}

This section aims to describe the combinatorial and geometric effect of chimney retractions on certain collections of minimal galleries in affine buildings.  In Section~\ref{sec:retractMin} we introduce gates with respect to chimneys, and use these to describe the effect of chimney retractions on arbitrary minimal galleries.  In Section \ref{sec:LFGs}, we introduce a notion of labeled folded galleries for arbitrary affine buildings, generalizing the definitions used in \cite{PRS, MST1, MNST}.  We then use this structure to establish bijections  between sets of minimal galleries and sets of labeled folded galleries.  The main result of this section is Theorem~\ref{thm:chimneyLFGs}, which is proved in Section \ref{sec:chimneysLFGs}, and places sets of labeled folded galleries in bijection with the intersection of preimages of the two main types of retractions considered in this paper.


\subsection{Effect of retractions on minimal galleries}\label{sec:retractMin}

In this section we give a precise description of the effect of chimney retractions on minimal galleries in $X$, in Proposition~\ref{prop:retractMinimal} below. We can then define a map from minimal galleries to positively folded galleries, in Corollary~\ref{cor:folding}.  To specify the effect of chimney retractions, we introduce the notion of a gate with respect to a chimney; see Definition~\ref{defn:gateChimney}.

Recall that $(\aW, \aS)$ is an irreducible affine Coxeter system with $[\overline{n}]$ the index set for~$\aS$, and that $X$ is an affine building of type $(\aW,\aS)$.  Recall that for any panel $p$ of $X$, we denote by $\star(p)$ the set of alcoves containing $p$, and that for any alcove $d$ of $X$, we denote by $\gate_d(p)$ the gate of $p$ with respect to $d$, as in Definition~\ref{defn:gate}.
 
The next result collects some useful relationships between gates, retractions, and half-apartments.  Given an alcove $d$ of the standard apartment $\App$, denote by $r_d: X \to \App$ the retraction of $X$ onto $\App$ centered at $d$.  If $d = \fa$, we write $r: X \to \App$.

\begin{lemma}\label{lem:gateRetractHalfApt}  
Let $d$ be an alcove of $\App$, let $p$ be a panel of $X$, and let $p'$ be the panel $r_d(p)$ of $\App$.  Let $H$ be the hyperplane of $\App$ supported by $p'$, let $H^-_d$ be the half-apartment of $\App$ which is bounded by $H$ and contains $d$, and let $H^+_d$ be the half-apartment of $\App$ which is bounded by $H$ and does not contain $d$.  

Then for all $c \in \star(p) \setminus \{ \gate_d(p)\}$, the alcoves $r_d(c)$ and $r_d(\gate_d(p))$ are the two distinct alcoves of $\App$ which contain the panel $p'$, with $r_d(c) \in H^+_d$ and $r_d(\gate_d(p)) \in H^-_d$.  In particular, $r_d(\gate_d(p)) = \gate_d(p')$, while for all $c \in \star(p) \setminus \{ \gate_d(p)\}$, we have $r_d(c) \neq \gate_d(p')$.
\end{lemma}
\begin{proof} These properties all follow from the fact that $r_d$ preserves the distance from any alcove of $X$ to $d$.
\end{proof}

There is also a close relationship between gates and minimal galleries in $X$, as follows.

\begin{lemma}\label{lem:gate}
Let $\gamma = (p_0,c_0, p_1, c_1, \dots, c_{l-1}, p_l,c_l,p_{l+1})$ be a gallery in $X$ which has the type of a minimal gallery.  Then $\gamma$ is minimal if and only if for all $1 \leq j \leq l$, we have $c_{j-1}=\gate_{c_0}(p_j)$ and $c_j \neq c_{j-1}$.
\end{lemma}

\begin{proof}  Since $\gamma$ has the type of a minimal gallery, it is minimal if and only if for all $1 \leq j \leq l$, we have $d(c_0, c_{j-1}) = j-1$  and $d(c_0,c_j) = j$.  Noticing that the alcoves $c_{j-1}$ and $c_{j}$ are both elements of $\star(p_j)$, consider $c,c' \in \star(p_j) \setminus \{ \gate_{c_0}(p_j) \}$.  As $\gate_{c_0}(p_j)$, $c$, and $c'$ all contain $p_j$, and $\gate_{c_0}(p_j)$ is the unique alcove in $\star(p_j)$ to be at minimal distance from $c_0$, we then have $d(c_0,\gate_{c_0}(p_j)) + 1 = d(c_0, c) = d(c_0,c')$.  Thus if $\gamma$ is minimal, we obtain $c_{j-1} = \gate_{c_0}(p_j)$ and $c_j \neq c_{j-1}$ for all $1 \leq j \leq l$, as required.

For the converse, we induct on the length of $\gamma$.  Suppose first that $\gamma = (p_0,c_0, p_1, c_1,p_2)$ with $c_1 \neq c_0$, and note that $\gate_{c_0}(p_1) = c_0$ since $c_0 \in \star(p_1)$.  Then it is clear that $\gamma$ is minimal.  Now suppose $\gamma$ is of length $l > 1$ with $c_{j-1} = \gate_{c_0}(p_j)$ and $c_j \neq c_{j-1}$ for all $1 \leq j \leq l$.  Then by induction, the initial subgallery $(p_0,c_0, p_1, c_1, \dots, c_{l-1},p_l)$ is minimal.  Hence $d(c_0,c_j) = j$ for all $1 \leq j \leq l-1$.  Now by the inductive assumption, $c_{l-1} = \gate_{c_0}(p_l)$ and $c_l \neq \gate_{c_0}(p_l)$, and thus from the previous paragraph we have $d(c_0,c_{l-1}) + 1 = d(c_0, c_l)$.  But then $d(c_0,c_l) = (l-1) + 1 = l$, and we conclude that $\gamma$ is minimal.
\end{proof}

To define gates with respect to chimneys, we will need the following result.

\begin{lemma}[Corollary~3.23 of~\cite{MNST}]\label{lem:bounded}  
Let $J \subseteq [n]$ and $y \in \aW$, and let $B$ be a bounded subset of $X$.  Then there is an alcove $d = d_B$ of $\App$ such that the retraction $r_{J,y}:X \to \App$ agrees on $B$ with the retraction $r_d: X \to \App$.
\end{lemma}

Since $\star(p)$ is a bounded subset for every panel $p$ of $X$, we may now use Lemmas~\ref{lem:gateRetractHalfApt} and~\ref{lem:bounded} to formulate gates with respect to chimneys in a well-defined manner. In particular, the gate of $p$ with respect to a chimney is independent of the choice of $d$ in the next definition. 

\begin{definition}\label{defn:gateChimney}  
Let $J \subseteq [n]$ and $y \in \aW$, and let $p$ be a panel of $X$.  Let $d$ be an alcove of $X$ such that the retraction $r_{J,y}$ agrees on $\star(p)$ with $r_d$.  The \emph{gate of $p$ with respect to the $(J,y)$-chimney}, denoted by $\gate_{J,y}(p)$, is the alcove $\gate_d(p)$.
\end{definition}

Now, using gates with respect to chimneys, we can give a precise description of the effect of a chimney retraction on a minimal gallery.  The next result generalizes and refines Lemma 2.9 in \cite{Hitzel}, which considers retractions from chambers at infinity.

\begin{prop}\label{prop:retractMinimal}  
Let $J \subseteq [n]$, let $y \in \aW$, and recall that $\phi_{J,y}$ denotes the orientation induced by the $(J,y)$-chimney.  Suppose \[ \gamma = (p_0,c_0, p_1, c_1, \dots c_{l-1}, p_l,c_l,p_{l+1})\] is a minimal gallery in~$X$, and let \[\gamma' = (p_0',c_0', p_1', c_1', \dots c_{l-1}', p_l',c_l',p_{l+1}')\] be image of $\gamma$ under the retraction $r_{J,y}: X \to \App$.
	  
Then $\gamma'=r_{J,y}(\gamma)$ has the same type as $\gamma$, is positively folded with respect to $\phi_{J,y}$, and for all $1 \leq j \leq l$ one has:
\begin{enumerate}
\item if $\gate_{J,y}(p_j) = c_{j-1}$, then $\gamma'$ has a positive crossing at $p_j'$;
\item if $\gate_{J,y}(p_j) = c_j$, then $\gamma'$ has a negative crossing at $p_j'$; and
\item if $\gate_{J,y}(p_j) \not \in \{ c_{j-1}, c_{i} \}$, then $\gamma'$ has a positive fold at $p_j'$.
\end{enumerate}
\end{prop}

\noindent Note that by Lemma~\ref{lem:gate}, since $\gamma$ is minimal, we have $c_{j-1} =  \gate_{c_0}(p_j)$ for $1 \leq j \leq l$.

\begin{proof}  Since the set of simplices in $\gamma$ forms a bounded subset of $X$, by Lemma~\ref{lem:bounded} there is an alcove $d = d_\gamma$ of $\App$ such that, for all $1 \leq j \leq l$, the retraction $r_{J,y}$ agrees on $\star(p_j)$ with $r_{d}$, and hence $\gate_{J,y}(p_j) = \gate_{d}(p_j)$.  Thus by Lemma~\ref{lem:gateRetractHalfApt}, the alcove $r_{d}(\gate_{d}(p_j))$ satisfies $\phi_{J,y}(r_{d}(\gate_{d}(p_j)), p'_j)=-$, while for all $c \in \star(p_j) \setminus \{ \gate_{d}(p_j) \}$, the alcove $r_{d}(c)$ satisfies $\phi_{J,y}(r_d(c), p'_j)=+$.

Since $\gamma$ is minimal, we have that $c_{j-1} \neq c_j$.  Thus at most one of $c_{j-1}$ and $c_j$ can equal $\gate_{d}(p_j)$.  From the previous paragraph, we see that if $c_{j-1} = \gate_{d}(p_j)$ then $\phi_{J,y}(c_{j-1}',p_j') = -$ while $\phi_{J,y}(c_j',p_j') = +$.  Hence $\gamma'$ has a positive crossing at $p_j'$, so (1) holds.  The proof of (2) is similar.  If neither $c_{j-1}$ nor $c_j$ is equal to $\gate_{d}(p_j)$, then $c_{j-1}' = r_{d}(c_{j-1}) = r_{d}(c_j) = c_j'$ and we have $ \phi_{J,y}(c_{j}',p_j') = +$.  Thus $\gamma'$ has a positive fold at $p_j'$, and (3) holds.

The gallery $\gamma'$ has the same type as $\gamma$ since the retraction $r_{J,y}$ is type-preserving, and $\gamma'$ is positively folded with respect to $\phi_{J,y}$ by (1)--(3).  
\end{proof}

We now write $\Gamma$ for the set of all minimal galleries in $X$, and $\Gamma_{J,y}^+$ for the set of all galleries in $\App$ which have the type of a minimal gallery and are positively folded with respect to the $(J,y)$-chimney.  Using the results of this section, we may define a ``folding function" from $\Gamma$ to $\Gamma_{J,y}^+$, as follows.

\begin{corollary}\label{cor:folding}  
The retraction $r_{J,y}: X \to \App$ induces a map 
\[
\cF_{J,y}: \Gamma \to \Gamma_{J,y}^+
\]
such that for each $\gamma = (p_0,c_0, p_1, \dots, c_l, p_{l+1})\in \Gamma$ and each $1 \leq j \leq l$, the gallery $\cF_{J,y}(\gamma)$:
\begin{enumerate}
\item has a positive crossing at $p_j'$, if $\gate_{J,y}(p_j) = c_{j-1}$;
\item has a negative crossing at $p_j'$, if $\gate_{J,y}(p_j) = c_j$; and
\item has a positive fold at $p_j'$, if $\gate_{J,y}(p_j) \not \in \{ c_{j-1}, c_j\}$.
\end{enumerate}
\end{corollary}

\noindent We note that, in general, the function $\cF_{J,y}$ is far from injective.


\subsection{Labelings and unfoldings}\label{sec:LFGs}

We have seen in Section~\ref{sec:retractMin} that applying a chimney retraction to a minimal gallery results in a positively folded gallery.  We now explain how to assign ``labels" to certain panels of positively folded galleries, so that each set of such labels is induced by a unique minimal gallery in its chimney retraction preimage.  We first give a careful formulation of the labelings that we will need, and then use these to construct a map from certain minimal galleries to certain labeled folded galleries, in Definition~\ref{def:label-fold-map}.  We prove in Proposition~\ref{prop:unfold} that any labeled folded gallery has a canonical ``unfolding" to a minimal gallery, and we deduce in Corollary~\ref{cor:bijection} that certain sets of minimal galleries and labeled folded galleries are in bijection.

We continue all notation from Section~\ref{sec:retractMin}, and for the remainder of this section, we fix $J \subseteq [n]$ and $y \in \aW$. We assume additionally from now on that $X$ is thick and panel-regular, as defined in Section~\ref{sec:buildings-prelim}.   For each $i \in [\overline{n}]$, let $q_i$ denote the unique panel of $\fa$ of type $i$. 

\begin{definition}\label{defn:labelings}  
Let $d$ be an alcove of $\App$.  A \emph{$(J,y,d)$-labeling} of $X$ consists of:
\begin{enumerate}
\item  For each $i \in [\overline{n}]$, we fix a \emph{labeling set} $\Lambda_i$ of cardinality $N_i = | \star(q_i) \setminus \{ \fa \}| \geq 2$.
\item  For each panel $p$ of $X$ of type $i$ we define the \emph{labeling set} $\Lambda_p$ to be equal to~$\Lambda_i$, and we fix a \emph{labeling bijection} 
\[ \Labl_{p,d}: \star(p) \setminus \{ \gate_d(p) \} \to \Lambda_p.\] 
\end{enumerate}
If $d = \fa$, then we simplify notation by referring to a \emph{$(J,y)$-labeling} instead of a $(J,y, \fa)$-labeling, and by defining $\Labl_p \define  \Labl_{p,\fa}$. 
\end{definition}

\noindent  
We note that, by construction, the labeling sets $\Lambda_p$ are independent of $J$, $y$ and $d$ (as is suggested by our notation).  However, the domain of the labeling bijection $\Labl_{p,d}$ depends on both $p$ and $d$.

\begin{definition} 
For each panel $p$ of $X$ and each alcove 
 $c \in \star(p) \setminus \{ \gate_d(p) \}$, we call $\Labl_{p,d}(c) \in \Lambda_p$ the \emph{local label of $c$ at $p$}. 
\end{definition}

\noindent We remark that if $p$ and $p'$ are distinct panels of a given alcove $c$, the local labels of $c$ at $p$ and at $p'$ are not required to be the same.  Indeed, as such panels $p$ and $p'$ necessarily have distinct types, the labeling sets $\Lambda_{p}$ and $\Lambda_{p'}$ 
may be distinct, and in general these two labeling sets will not even have the same cardinality.  

We now generalize the definitions of labeled folded alcove walks from \cite{PRS} and \cite{MNST}.  We discuss the relationship between our definitions and these earlier ones in Remark~\ref{rem:comparisonLF}.

\begin{definition}\label{def:labeledGalleries}  Let $c_0$ be an alcove of $\App$, and fix a $(J,y,c_0)$-labeling of $X$.  Let $\gamma = (p_0,c_0, p_1, c_1, \dots c_{l-1}, p_l,c_l,p_{l+1})$ be a gallery in $\App$, and denote by $\cP_\gamma$ the set of panels $p_j$ at which $\gamma$ has either a positive fold or a positive crossing with respect to $\phi_{J,y}$. Let  $\Labl_\gamma$ be a map that assigns to every panel $p_j$ in $\cP_\gamma$ an element $\Labl_\gamma(p_j)\in \Lambda_{p_j}$. 

We say that the pair $(\gamma, \Labl_\gamma)$ is a \emph{$(J,y,c_0)$-labeled folded gallery} if:
	\begin{enumerate}
		\item $\gamma$ has the type of a minimal gallery;
		\item $\gamma$ is positively folded with respect to $\phi_{J,y}$; and
		\item if $\gamma$ has a positive fold at $p_j$, then 
		\[\Labl_\gamma(p_j)\in\Lambda_{p_j}\setminus\{ \Labl_{p_j,c_0}(\gate_{J,y}(p_j))\}, \]
		where $\Labl_{p_j,c_0}$ is the labeling bijection from Definition~\ref{defn:labelings}. 
	\end{enumerate}
If $c_0 = \fa$, we will say that $(\gamma,\Labl_\gamma)$ is a \emph{$(J,y)$-labeled folded gallery}.  When $J$, $y$, and $c_0$ are clear, we may refer to $(\gamma, \Labl_\gamma)$ as just a \emph{labeled folded gallery}. We may also suppress the map $\Labl_\gamma$ and speak of a labeled folded gallery $\gamma$. 
\end{definition}

For any alcove $c_0$ of $\App$, we now denote by $\LG_{J,y,c_0}^+$ the set of all $(J,y,c_0)$-labeled folded galleries; note that such galleries necessarily have first alcove $c_0$.  As we will show, such labeled folded galleries will naturally arise as images of minimal galleries in $X$ under retractions, and will turn out to have unique lifts to $X$.  Note that the label $\Labl_\gamma(p_j)$ of a panel $p_j$ with a positive crossing is an element of $\Lambda_{p_j}$, while the label of a panel $p_j$ with a positive fold avoids the label of the gate of $p_j$ with respect to the $(J,y)$-chimney, by item (3) of the definition.  The reason for not labeling the panels at which $\gamma$ has negative crossings is that, as will be seen in the proof of Proposition~\ref{prop:unfold} below, there will be no choice in how to ``lift" such crossings.

\begin{remark}\label{rem:comparisonLF}  
We wish to emphasize the distinction between the labelings on positively folded galleries from \cite{PRS}, which are also used in the authors' subsequent work \cite{MNST,MST2}, and the labeling scheme from Definition \ref{def:labeledGalleries}.  Let $G$ be a connected reductive group over $F = \kk((t))$, where $\kk$ is any field. Then $G(F)$ has an associated affine building, namely its Bruhat--Tits building.  In this special case, the labels on a positively folded gallery are given by individual elements of the residue field $\kk$ of $F$ (resp.~its invertible elements~$\kk^{\times}$) at every positive crossing (resp.~fold), and Definition~\ref{def:labeledGalleries} recovers the labeling from \cite{PRS}; compare \cite[Def.~5.9]{MNST}.

If $U_{\alpha,j}$ denotes the subgroup of $G(F)$ which fixes the half-apartment $\alpha^{j,\id}$, for $\alpha \in \Phi$ and $j \in \Z$, then the quotient $\overline{U}_{\alpha, j} = U_{\alpha,j} / U_{\alpha, j+1}$ can be naturally identified with a one-parameter ``root subgroup'' of $G(F)$ which is isomorphic to the additive group~$(\kk,+)$.  In this function field setting, the entire group $G(F)$ is in fact generated by its subgroups~$\overline{U}_{\alpha,j}$, and images of labeled folded galleries under chimney retractions can thus be encoded purely algebraically using relations in $G(F)$ involving multiplication by elements of~$\overline{U}_{\alpha,j}$.  Due to this direct correspondence between labels on galleries and the algebra in the group $G(F)$, it is not required in \cite{MNST}, for example, to view labeled folded galleries geometrically as images of chimney retractions, as we do below in Definition~\ref{def:label-fold-map}.

In our more general setting of arbitrary affine buildings, the labeling set $\Lambda_i$ is any set in bijection with $\star(q_i) \setminus \{ \fa \}$, and more significantly, the analogous one-parameter ``root subgroups'' with elements in bijection with labeling sets are no longer guaranteed.
   For instance, the group $\SL_n(\Q_p)$ has Bruhat--Tits building in which each panel is contained in $p+1$ alcoves.   So we could, if we like, use the elements of $\F_p$ (resp.~$\F_p^\times$) to label positive crossings (resp. folds).  The subgroups $U_{\alpha,j}$ of $\SL_n(\Q_p)$ can still be defined as fixators of half-apartments, as in the function field case, and their quotients $\overline{U}_{\alpha,j}$ are each isomorphic to the additive group $(\F_p,+)$.  However, the key difference is that~$\overline{U}_{\alpha,j}$ cannot be identified with any subgroup of~$\SL_n(\Q_p)$.  Thus, out of necessity, the bijections we obtain between our more general labeled folded galleries and preimages of chimney retractions have purely geometric proofs, which are decidedly different from those appearing in any of the previously mentioned works.  Compare the algebraic proof of \cite[Thm.~5.11]{MNST} to the geometric proof of the analogous Theorem \ref{thm:chimneyLFGs} in Section \ref{sec:chimneysLFGs}, for example.
\end{remark}

Now recall from Corollary~\ref{cor:folding} that the retraction $r_{J,y}:X \to \App$ induces a function $\cF_{J,y}: \G \to \G^+_{J,y}$ from minimal galleries in $X$ to positively folded galleries in $\App$.    For any alcove $c_0'$ of $X$, denote by $\G_{c_0'}$ the set of all minimal galleries with first alcove~$c_0'$, and for any alcove $c_0$ of $\App$, define
\[
\widetilde\G_{c_0}:= \bigcup_{c_0'\in r_{J,y}^{-1}(c_0)}\G_{c_0'}.
\]
We now enhance $\cF_{J,y}$ to a ``labeling and folding"\, function.

\begin{definition}\label{def:label-fold-map} For any alcove $c_0$ of $\App$, define a map
	\[
	\LablF_{J,y,c_0}: \widetilde\G_{c_0} \to \LG^+_{J,y,c_0}
	\]
	which assigns to a gallery $\gamma' \in \widetilde\G_{c_0}$ a labeled folded gallery $(\gamma, \Labl_{\gamma}) \in \LG^+_{J,y,c_0}$ as follows.
	For every $c_0'\in r^{-1}_{J,y}(c_0)$ and every $\gamma' = (p_0', c_0', p_1',\dots,c_l',p_{l+1}') \in \G_{c_0'}$, define 
	\[\gamma\define\cF_{J,y}(\gamma')\] and for $1 \leq j \leq l$, if $\gamma$ has a positive crossing or a positive fold at $p_j$, put 
	\[ \Labl_{\gamma}(p_j)\define\Labl_{p_j',c_0'}(c_j').\]
\end{definition}

\noindent By Corollary~\ref{cor:folding} and Definition~\ref{defn:labelings}, the map $\Labl_{\gamma'}$ satisfies conditions (1)--(3) in Definition~\ref{def:label-fold-map}, and the labeling and folding map $\LablF_{J,y,c_0}$ is hence well-defined. 

Our next goal is to show that for any $c_0' \in r_{J,y}^{-1}(c_0)$, the restriction of $\LablF_{J,y,c_0}$ to $\G_{c_0'}$ is a bijection.  For this, we will describe precisely how labeled folded galleries in $\App$ may be ``unfolded" to minimal galleries in the thick, panel-regular affine building $X$.  The next result generalizes and refines Proposition 3.3 of~\cite{Hitzel}, which shows that any positively folded gallery has some minimal preimage with the same first alcove,  under a retraction from a chamber at infinity.

\begin{prop}\label{prop:unfold}  
Let  $J\subseteq [{n}]$ and $y \in \aW$, and let $c_0$ be an alcove of $\App$. Fix a $(J,y,c_0)$-labeling on $X$ and let $(\gamma, \Labl_\gamma)$ be a $(J,y,c_0)$-labeled folded gallery in $\App$  with 
\[\gamma = (p_0,c_0, p_1,c_1, \dots, c_{l-1}, p_l,c_l,p_{l+1}).\]
Then for every $c_0'\in r^{-1}_{J,y}(c_0)$ there is a unique gallery 
\[\gamma' = (p_0',c_0', p_1', c_1',\dots, c_{l-1}', p_l',c_l',p_{l+1}')\] 
in $X$ such that:
\begin{enumerate}
\item $\gamma'$ is minimal;
\item $r_{J,y}(\gamma') = \gamma$;
\item if $\gamma$ has a positive crossing at $p_j$, then $\Labl_{p_j',c_0'}(c_j') = \Labl_\gamma(p_j) \in \Lambda_{p_j}$ and $c_j' \neq \gate_{c_0'}(p_j')$; and
\item if $\gamma$ has a positive fold at $p_j$, then $\Labl_{p_j',c_0'}(c_j')= \Labl_\gamma(p_j) \in \Lambda_{p_j}\setminus\{\Labl_{p_j',c_0'}(\gate_{J,y}(p_j'))\}$ and $c_j' \not \in \{ \gate_{c_0'}(p_j'), \gate_{J,y}(p_j')\}$.
\end{enumerate}
\end{prop}
\begin{proof}  By definition of labeled folded galleries, $\gamma$ has the type of a minimal gallery.  Since $p_0$ is a face of $c_0$ and $p_{l+1}$ is a face of $c_l$, it thus suffices to consider $\gamma$ of the form $\gamma = (c_0, p_1, \dots, c_{l-1},p_l,c_l)$.  

We now induct on the length of $\gamma$.  Consider a labeled folded gallery $\gamma = (c_0, p_1, c_1)$, and choose $c_0' \in r_{J,y}^{-1}(c_0)$.  Define $p_1'$ to be the panel of $c_0'$ of the same type as $p_1$; this is a necessary condition for $r_{J,y}(\gamma') = \gamma$.  Now as $\gate_{c_0'}(p_1') = c_0'$, by Lemma~\ref{lem:gate} any gallery $(c_0', p_1', c_1')$ where $c_1' \neq c_0'$ will be minimal.  

Suppose first that $\gamma$ has a negative crossing at $p_1$.  Then $c_1 \neq c_0$, and we have $c_1 = \gate_{J,y}(p_1)$.  Define $c_1' = \gate_{J,y}(p_1')$.  Then $r_{J,y}(c_1') = \gate_{J,y}(p_1) = c_1$, and so $r_{J,y}(\gamma') = \gamma$ for $\gamma' = (c_0', p_1', c_1')$.  If $c_1' = c_0'$ then $c_1 = r_{J,y}(c_0') = c_0$, a contradiction.  Thus $\gamma'$ is minimal.  Properties (3) and (4) hold trivially in this case.  For uniqueness, we observe that if $\gamma'' = (c_0',p_1', c_1'')$ is minimal and $c_1'' \neq c_1'$, then $c_1'' \neq \gate_{J,y}(p_1')$.  But then $r_{J,y}(c_1'') \neq c_1$, and so $r_{J,y}(\gamma'') \neq \gamma$.  Hence $\gamma'$ is unique.

Next, suppose $\gamma$ has a positive crossing at $p_1$.  Note that  $c_1 \neq c_0$ in this case too.  Let $c_1'$ be the unique element of $\star(p_1') \setminus \{ \gate_{c_0'}(p_1') \} = \star(p_1') \setminus \{ c_0' \}$ such that $ \Labl_{p_1'}(c_1') = \Labl_\gamma(p_1)$ (it is possible that $c_1' = c_1$ here, if $c_0' = c_0$).   Then $\gamma' = (c_0', p_1', c_1')$ is minimal and property (3) holds.  Now since $\gamma$ has a positive crossing at $p_1$, we have $\gate_{c_0}(p_1) = \gate_{J,y}(p_1) = c_0$.  As $c_1 \neq c_0$ and $c_1' \neq c_0'$, it follows that $r_{J,y}(c_1') = c_1$.  Hence $r_{J,y}(\gamma') = \gamma$.  Property (4) holds trivially in this case, and the uniqueness of $\gamma'$ follows from the unique choice of $c_1'$ (and of $p_1'$).

Now suppose $\gamma$ has a positive fold at $p_1$.  Then $c_1 = c_0 = \gate_{c_0}(p_1)$ and $\gate_{J,y}(p_1) \neq c_0$, and we have $c_0' = \gate_{c_0'}(p_1')$.  If $\gate_{J,y}(p_1') = \gate_{c_0'}(p_1')$, then $r_{J,y}(c_0') = c_0 = \gate_{J,y}(p_1)$, a contradiction.  Hence $\gate_{J,y}(p_1') \neq \gate_{c_0'}(p_1')$.   So we may let $c_1'$ be the unique element of $\star(p_1') \setminus \{ \gate_{c_0'}(p_1'), \gate_{J,y}(p_1') \}$ such that $\Labl_{p_1'}(c_1') = \Labl_\gamma(p_1)$ (note that $c_1' \neq c_1$ here, even if $c_0' = c_0$).  Then $\gamma' = (c_0', p_1', c_1')$ is minimal.  Since $c_1' \neq \gate_{J,y}(p_1')$ and $c_1 \neq \gate_{J,y}(p_1)$, we get that $r_{J,y}(c_1') = c_1$.  Thus $r_{J,y}(\gamma') = \gamma$.  Property (3) holds trivially in this case, and the uniqueness of $\gamma'$ follows from the unique choice of $c_1'$ (and of $p_1'$).

We now consider $\gamma$ of length $l > 1$.  By induction, given $c_0' \in r_{J,y}^{-1}(c_0)
$, there is a unique gallery $\gamma'_{l-1} = (c_0', p_1', c_1', \dots, c_{l-1}')$ such that $\gamma'_{l-1}$ is minimal, $r_{J,y}(\gamma'_{l-1}) = (c_0, p_1, c_1, \dots, c_{l-1})$, and (3) and (4) hold for $1 \leq j \leq l-1$.  Define $p_l'$ to be the panel of $c_{l-1}'$ of the same type as $p_l$.  Then since $\gamma$ has the type of a minimal gallery, there is an alcove $c_l''$ containing $p_l'$ such that the gallery $(c_0', p_1', c_1', \dots, c_{l-1}',p_l',c_l'')$ is minimal.  Hence by Lemma~\ref{lem:gate}, $c_{l-1}' = \gate_{c_0'}(p_l')$, and any gallery $\gamma' = (c_0', p_1', c_1', \dots, c_{l-1}',p_l',c_l')$ with $c_l' \neq c_{l-1}'$ will be minimal.  Moreover, to see that $r_{J,y}(\gamma') = \gamma$, it suffices to verify that $r_{J,y}(c_l') = c_l$.  We may now in each case construct a unique $c_l' \in \star(p_l') \setminus \{ c_{l-1}'\}$ such that $r_{J,y}(c_l') = c_l$ and (3) and (4) hold, using similar arguments to those for $l = 1$. This completes the proof.
\end{proof}

Recall that $\G_{c_0'}$ is the set of all minimal galleries in $X$ with first alcove $c_0'$, and $\LG^+_{J,y,c_0}$ is the set of all $(J,y,c_0)$-labeled folded galleries (which necessarily have first alcove $c_0$).
Let $\Labl\cF_{J,y,c_0,c_0'}$ denote the restriction of the map $\Labl\cF_{J,y,c_0}:\tilde \G_{c_0} \to \LG^+_{J,y,c_0}$ (from Definition~\ref{def:label-fold-map}) to the set $\G_{c_0'}$.

\begin{corollary}\label{cor:bijection}
For any alcove $c_0$ of $\App$ and any $c_0' \in r_{J,y}^{-1}(c_0)$, the function \[\Labl\cF_{J,y,c_0,c_0'}: \G_{c_0'} \to \LG^+_{J,y,c_0}\] is a bijection. 
\end{corollary}
\begin{proof}  The procedure described in Proposition~\ref{prop:unfold} provides an ``unfolding" function, say $\cU_{J,y,c_0,c_0'}:\LG^+_{J,y,c_0} \to \G_{c_0'}$.   It is then straightforward to verify that $\cU_{J,y,c_0,c_0'}$ is a (two-sided) inverse to $\Labl\cF_{J,y,c_0,c_0'}$.
\end{proof}

If $c_0 = c_0' = \fa$, we write $\Labl\cF_{J,y}$ for $\Labl\cF_{J,y,\fa,\fa}$ and $\LG^+_{J,y}$ for $\LG^+_{J,y,\fa}$, and we have the following important special case.

\begin{corollary}\label{cor:bijection_fa} 
The function $\Labl\cF_{J,y}: \G_\fa \to \LG^+_{J,y}$ is a bijection.
\end{corollary}


\subsection{Retractions and labeled folded galleries}\label{sec:chimneysLFGs}

We now prove our main geometric result, Theorem~\ref{thm:chimneyLFGs}, which is restated below.   Recall that $r:X \to \App$ is the retraction onto $\App$ centered at the base alcove $\fa$.

\chimneyLFGsRestate*

\begin{proof}
	Fix a minimal gallery $\gamma_x$ of type $\vec{x}$ from $\fa$ to $\x$ in the base apartment $\App$.  By definition, the preimage $r^{-1}(\gamma_x)$ is the set, which we denote by $\G_\fa(\gamma_x)$, of all minimal galleries in $X$ with first alcove $\fa$ and the same type as $\gamma_x$.  Therefore, the set $r^{-1}(\x) \cap r^{-1}_{J,y}(\z)$ is in bijection with the set of all galleries in $\G_\fa(\gamma_x)$ whose final alcove is mapped onto $\z$ by $r_{J,y}$; denote this set by $\G_\fa(\gamma_x,\z')$.  
	
	Now write $\LG_{J,y}^+(\gamma_x,\z)$ for the set of all $(J,y)$-labeled folded galleries of type $\gamma_x$ with final alcove $\z$.  By definition of the bijection $\Labl\cF_{J,y}: \G_\fa \to \LG_{J,y}^+$ from Corollary~\ref{cor:bijection_fa}, any $\gamma' \in \G_\fa$ with final alcove $c_l'$ is mapped to an element of $\LG_{J,y}^+$ with final alcove $r_{J,y}(c_l')$.  Therefore $\Labl\cF_{J,y}$ restricts to an injective map from $\G_\fa(\gamma_x,\z')$ into $\LG_{J,y}^+(\gamma_x,\z)$.  Moreover, given any $\gamma \in \LG_{J,y}^+(\gamma_x,\z)$, the unique preimage $\gamma'$ of $\gamma$ under the bijection $\Labl\cF_{J,y}$ also has type $\gamma_x$, and has final alcove in $r^{-1}_{J,y}(\z)$.  That is, $\gamma' \in \G_\fa(\gamma_x,\z')$.  Therefore $\Labl\cF_{J,y}$ induces a bijection from  $\G_\fa(\gamma_x,\z')$ to $\LG_{J,y}^+(\gamma_x,\z)$, as required.
\end{proof}


\section{Shadows of chimney retractions and orbits in affine flags}\label{sec:chimney-shadows}

In this final section, we link the geometric and combinatorial structures associated with buildings  introduced in Sections~\ref{sec:Chimneys} and \ref{sec:retractions} to algebraic properties of the defining groups in the Bruhat--Tits case.  
The main goal is to prove Theorems \ref{thm:DoubleCosetsI-intro} and \ref{thm:DoubleCosetsK-intro} from the introduction, as well as the more general Theorem \ref{thm:DoubleCosetsPar}.

We review the algebraic context for groups which admit affine buildings in Section \ref{sec:groups-prelim} and establish notation for the associated spherical parabolic subgroups in Section \ref{sec:spherical}.
Preimages of chimney retractions are then related to intersections of double cosets in affine flag varieties as shown in Propositions \ref{prop:inverse_classical} and \ref{prop:inverse_chimney}, as well as Corollary \ref{cor:inverse_chimney}. These are the key results in Section~\ref{sec:classic-retract}, which only relies on the material in Section~\ref{sec:Chimneys}. The results of Section~\ref{sec:retractions} are applied in Section~\ref{sec:shadow-retract}, where the connection between labeled folded galleries and the algebraic framework is made.  
The proofs of Theorems~\ref{thm:DoubleCosetsI-intro} and~\ref{thm:DoubleCosetsK-intro} then follow by applying Proposition \ref{prop:DoubleCosetsI-body} and Corollary \ref{cor:DoubleCosetsK-body}, respectively, combined with Theorem \ref{thm:nonempty}, all of which appear in Section \ref{sec:shadow-retract}.


\subsection{Groups with affine Tits systems}\label{sec:groups-prelim}

In the next two subsections, we review the algebraic situation when a group $G$ has an associated affine building; see Chapter 6 of \cite{AB} for more details. We maintain all notation previously established, giving additional algebraic meaning to some of the geometric terminology from Section \ref{sec:preliminaries}.

Let $G$ be a group which admits an \emph{affine Tits system} $(G, I, N, \aS)$.  This means that $I$ and $N$ are subgroups of $G$ such that $G = \langle I, N \rangle$, the intersection $ I \cap N$ is normal in $N$, the quotient $\aW = N / (I \cap N)$ is generated by $\aS$, the pair $(\aW,\aS)$ is an affine Coxeter system, and certain additional axioms hold (see, for example, Definition~6.55 of~\cite{AB}).  We call $\aW$ the \emph{affine Weyl group} of $G$.  An affine Tits system is the same thing as an \emph{affine $BN$-pair}, but we employ the terminology of Tits systems since we will use $B$ for a different subgroup of $G$ in the next section.  

Let $X$ be the building associated to this Tits system.  Then $X$ is an affine building of type $(\aW,\aS)$, and hence can be realized as a simplicial complex as described in Section~\ref{sec:preliminaries}.   Recall that $\App$ denotes the standard apartment of $X$ and $\fa$ the fundamental or base alcove of $X$.  Then by construction, the group $G$ has a type-preserving and strongly transitive action on $X$,  the subgroup $I \cap N$ of $G$ fixes $\App$ pointwise, $N$ stabilizes $\App$ setwise and acts transitively on the set of alcoves in $\App$, and $I$ is the stabilizer of $\fa$ in~$G$.  We will denote by $K$ the subgroup of $G$ which stabilizes the origin.

To provide an example which motivates our notation, if $G$ is a simple, simply-connected reductive group and $F$ is any local field with ring of integers $\cO$, then $G(F)$ admits an affine Tits system $(G(F), I, N, \aS)$ where $I \leq G(\cO)$ is the \emph{Iwahori subgroup} of $G(F)$, $I \cap N = T(\cO)$ for $T$ a maximal torus, $(\aW,\aS)$ is the affine Weyl group of $G(F)$, $K$ is the \emph{parahoric subgroup} $K = G(\cO)$, and the associated building $X$ is the Bruhat--Tits building for $G(F)$. (While $G$ here is simply-connected, our geometric results extend naturally to all connected reductive groups by working with \emph{extended alcoves} as in \cite{MST2}.) All affine Kac--Moody groups also admit affine Tits systems.


\subsection{Spherical parabolic subgroups}\label{sec:spherical}

Now for any $G$ with an affine Tits system, let $B$ be the stabilizer in $G$ of the \emph{antidominant Weyl chamber} $\mathcal{C}_{w_0}$ of $\App$.  Then $G$ also has a spherical Tits system $(G, B, N, \sS)$ with the same $N$ as above, and the associated building is the spherical building at infinity of the affine building $X$.  The \emph{spherical} or \emph{finite Weyl group} of $G$ is the quotient $\sW = N/(B \cap N)$, with canonical generating set $\sS$.  We may factor $B$ as $B = TU$ where $T = B \cap N$.  If $G$ is a simple, simply-connected reductive group and $F$ a local field, then $B = B(F)$ is a \emph{Borel subgroup} of $G(F)$ containing $T = T(F)$, and $U = U(F)$ is a \emph{unipotent subgroup}.

We denote by $B^{w_0} = w_0 B w_0^{-1}$ the stabilizer of the \emph{dominant Weyl chamber} $\Cf$.  For any subgroup $H$ of $G$ and any $g \in G$, we similarly denote the conjugate subgroup by $H^g = gHg^{-1}$. We identify $\App$ with $V^*$ such that the set of positive roots $\Phi^+$ are those which are positive with respect to $B^{w_0}$. Under this convention, the roots in $\Phi$ which lie in the half-apartment of $\App$ bounded by $H_{\widetilde\alpha}$ and containing $\Cf$ are precisely those roots in $\Phi^+$, and for $\lambda \in R^\vee$, the element $t^\lambda \in \aW$ acts on $\App$ by translating by $\lambda$ (rather than by $-\lambda$, had we chosen $\Phi^+$ to be those roots which are positive with respect to $B$).

Let $P$ be a \emph{standard spherical parabolic subgroup} of $G$, meaning a subgroup of $G$ which contains $B$.  The set of standard spherical parabolic subgroups is in bijection with the subsets $J \subseteq [n]$, and so we often write $P = P_J$.  Then $P_J$ admits a \emph{Levi decomposition} $P_J = L_J U_J$, where $L_J$ is the Levi component and $U_J$ is its unipotent radical.  The Levi subgroup $L_J$ also admits an affine and a spherical Tits system. Its finite Weyl group is the spherical Coxeter system $(W_J, S_J)$, and it has a corresponding basis of simple roots $\Delta_J  \subseteq \Delta$, a set of positive roots $\Phi_J^+ = \Phi^+ \cap \Phi_J$, and so on. (Note that in the spherical context, we omit the 0-subscript in $(\sW)_J$ to avoid double indices.) In particular, standard spherical parabolic subgroups $P_J$ are in bijection with Levi subgroups $L_J$, as well as with subsets $J \subseteq [n]$, and we use all three methods of indexing objects by an associated parabolic. The subgroups \[ I_P = I_{P_J} := (I\cap L_J)U_J\] and their conjugates $(I_P)^y=y I_P y^{-1}$ for $y \in \aW$ play a prominent role in the remainder of the paper.


\subsection{Preimages of chimney retractions are orbits in affine flag varieties}\label{sec:classic-retract}

From now on, let $G$ be a group with an affine Tits system, and let $X$ be the associated affine building with standard apartment $\App$. In this section, we express preimages of chimney retractions as orbits of certain subgroups of $G$. We emphasize that the material in Section \ref{sec:classic-retract} does not make use of Section~\ref{sec:retractions}.

Recall that $r_{J,y}: X \to \App$ denotes the retraction from the $(J,y)$-chimney, and that we simply write $r_J: X \to \App$ in the case where $y=1$.  We first discuss preimages under the classical retractions.   If $J = [n]$, then $r_J$ is the retraction onto $\App$ centered at the base alcove $\fa$, which we denote simply by $r: X \to \App$. If $J=[n]$ and $y \in \aW$ is arbitrary, we write $r_{\y}:X \to \App$ for the retraction onto $\App$ centered at the alcove $\y = y \fa$.  If $J = \emptyset$, then $r_\emptyset$ is the retraction onto $\App$ centered at the chamber at infinity represented by  $\Cw_{w_0}$, since by our conventions $B$ is the stabilizer in $G$ of the antidominant Weyl chamber~$\Cw_{w_0}$.  (Note that the retraction $r_\emptyset$ is sometimes written as $r_U$, since $B = TU$ and so the chamber at infinity represented by $\Cw_{w_0}$ is stabilized by $U$.)

By the construction of $X$, and since $I$ is the stabilizer in $G$ of the base alcove $\fa$, for all $x \in \aW$ we may identify the alcove $\mathbf{x} = x\fa$ of $\App$ with the left coset $xI$.  For $\lambda \in R^\vee$, we may similarly identify the vertex $\lambda$ of $\App$ with the left coset $t^\lambda K$.  These identifications are used in the following classical result characterizing the preimages of $r$ and $r_\emptyset$.

\begin{prop}\label{prop:inverse_classical}
Let $r:X \to \App$ be the retraction onto $\App$ centered at the base alcove~$\fa$, and let $r_\emptyset: X\to \App$ be the retraction on $\App$ from the chamber at infinity represented by~$\Cw_{w_0}$.
\begin{enumerate}
\item For all $x \in \aW$, we have $r^{-1}(\x) = I x I$ and $r_\emptyset^{-1}(\x) = U x I$.
\item For all $\lambda \in R^\vee$, we have $r^{-1}(\sW \cdot \lambda) = K t^\lambda K$ and $r_\emptyset^{-1}(\sW \cdot \lambda) = U t^\lambda K$. 
\end{enumerate}
\end{prop}

\begin{proof} For (1), see Section 3.4 of \cite{GaussentLittelmann} or Sections 6--7 of \cite{BruhatTits}, for example. Using that $K = \sqcup_{w \in \sW} IwI$ so that $Kt^\lambda K = \sqcup_{v,w \in \sW} I t^{v\lambda}w I$ and $Ut^\lambda K = \sqcup_{w \in \sW} U t^\lambda w I$, the equalities in (2) then follow directly from (1).
\end{proof}

We next observe the following consequence of Proposition~\ref{prop:inverse_classical}.   For any face $\sigma$ of $\fa$ which contains the origin $v_0$, we denote by $K(\sigma)$ the parahoric subgroup of $G$ which is the stabilizer in $G$ of $\sigma$. For example, $K = K(v_0)$ and $I = K(\fa)$.  We write $W_\sigma$ for the subgroup of $\sW$ which stabilizes $\sigma$. (To avoid double indices, we omit the 0-subscript here.) Thus, $W_\sigma$ is a subgroup of $\sW$ which is also contained in $K(\sigma)$. 

\begin{corollary}\label{cor:inverse_classical}   For any faces $\sigma$ and $\tau$ of the base alcove $\fa$ which contain the origin and any $x \in \aW$, we have $r^{-1}(W_\sigma \cdot x K(\tau)) = K(\sigma) x K(\tau)$.
\end{corollary}

\begin{proof}  This follows from the preimage of $r$ given by Proposition~\ref{prop:inverse_classical}(1), together with the fact that the retraction $r$ is type-preserving.
\end{proof}

We now generalize Proposition~\ref{prop:inverse_classical} to arbitrary chimney retractions.  As with our treatment of chimney retractions in~\cite{MNST}, much of this discussion will be a formalization and generalization of arguments from Section 11.2 of~\cite{GHKRadlvs}.  

We first state a Bruhat decomposition result.  This lemma is established in~\cite{GHKRadlvs} in the case $G = G(\overline{\F}_q((t)))$, where $G$ is a split connected reductive group over $\overline{\F}_q$ an algebraic closure of the finite field~$\F_q$. However, the proof immediately extends to any $G$ with an affine Tits system.  

\begin{lemma}[Lemma~11.2.1 of \cite{GHKRadlvs}]\label{lem:Bruhat} Let $P$ be a standard spherical parabolic subgroup of $G$.  Then 
\[ G = \bigsqcup_{z \in \aW} I_P z I. \]
\end{lemma}

The proof of the next result also extends immediately from~\cite{GHKRadlvs} to our setting.   

\begin{lemma}[Lemma~11.2.2 of~\cite{GHKRadlvs}]\label{lem:fixAlcoves}  Let $P = P_J = L_J U_J$ be a standard spherical parabolic subgroup of $G$.  Let $\lambda \in R^\vee \cap \left( \cap_{\alpha \in \Phi_J^+} H_{\alpha,0} \right)$, and let $v \in \sW$ be a minimal length representative of a (right) coset in $W_J \backslash \sW$.  Then:
\begin{enumerate}
\item Every element of the group $I \cap L_J$ fixes the alcove $t^\lambda v \fa$.
\item Let $g \in U_J$.  If $\lambda \in \cap_{\beta \in \Phi^+ \setminus \Phi_J^+} \beta^{n_\beta,w_0}$ for some collection of integers \[\{ n_\beta \in \Z \mid \beta \in \Phi_J \setminus \Phi_J^+\}\] depending upon $g$, then $g$ fixes the alcove $t^\lambda v \fa$.
\end{enumerate}
\end{lemma}

We now translate Lemma~\ref{lem:fixAlcoves} into a statement about sectors.  Given $P = P_J = L_J U_J$ as in Lemma~\ref{lem:fixAlcoves}, it will be convenient to refer to the region $\bigcap_{\alpha \in \Phi_J^+} \left(\alpha^{0,\id} \cap \alpha^{1,w_0} \right)$ of $\App$ as the \emph{$J$-corridor}.

\begin{corollary}\label{cor:fixSector1}  Let $P = P_J = L_J U_J$ be a standard spherical parabolic subgroup of $G$, and let $y \in \aW$.  Then:
\begin{enumerate}
\item Every element of the group $(I \cap L_J)^y$ fixes every $(J,y)$-sector.
\item Let $g \in U_J^y$.  Then there is a $(J,y)$-sector which is fixed by $g$.
\end{enumerate}
\end{corollary}
\begin{proof}  We may assume $y = \id$ for both statements.  If $J = [n]$ then $I \cap L_J = I$ and $U_J$ is trivial, while the unique $J$-sector is the base alcove $\fa$, so the result is clear.  Next suppose $J = \emptyset$, equivalently $P = B$.   Then $I \cap L_J = I \cap T$ and $N = U$.  Hence $I \cap L_J$ fixes $\App$, since $T$ does, and so (1) holds in this case.  For (2), we have that $U$ stabilizes the chamber at infinity corresponding to the antidominant Weyl chamber~$\mathcal{C}_{w_0}$.  It follows that every element of $U$ fixes some subsector of the antidominant Weyl chamber.  Since every subsector of $\mathcal{C}_{w_0}$ is a $J$-sector with $J = \emptyset$, we obtain~(2).

We may now assume that $J$ is a nonempty proper subset of $[n]$.  We will show that in this case, $I \cap L_J$ fixes the entire $J$-corridor.  Since any $J$-sector is contained in the $J$-corridor, (1) follows.  Now by Lemma~\ref{lem:fixAlcoves}(1), the group $I \cap L_J$ fixes all alcoves of the form $t^\lambda v \fa$ where $\lambda \in R^\vee \cap \left( \cap_{\alpha \in \Phi_M^+} H_{\alpha,0} \right)$ and $v$ is a minimal length representative for a coset in $W_J \backslash \sW$.  Observe that these alcoves all lie in the $J$-corridor.  

Now let $\x$ be any alcove in the $J$-corridor, and let $\gamma = \gamma_{\lambda,v}$ be a minimal gallery from some (fixed) alcove $t^\lambda v \fa$ as in the statement of Lemma~\ref{lem:fixAlcoves} to $\x$.  We will show that $I \cap L_J$ fixes $\x$ by induction on the number of alcoves in $\gamma$.  The previous paragraph gives the result when $\gamma$ has only one alcove, namely $\x = t^\lambda v\fa$.  Now suppose $\gamma$ has $j \geq 2$ alcoves, and let $\x'$ be the second last alcove in $\gamma$, so that $\x' \neq \x$ but $\x'$ is adjacent to $\x$.  Since the $J$-corridor is obtained by intersecting half-apartments of $\App$, any minimal gallery between two alcoves in the $J$-corridor is also contained in the $J$-corridor.  Hence $\x'$ is contained in the $J$-corridor.  Now as the initial subgallery of $\gamma$ from $t^\lambda v\fa$ to $\x'$ has $j - 1$ alcoves, by induction $\x'$ is fixed by $I \cap L_J$.  

Let $p$ be the panel between $\x'$ and $\x$, let $H$ be the hyperplane of $\App$ supported by $p$, and let $H^\x$ be the half-apartment of $\App$ which is bounded by $H$ and contains $\x$.  As the distinct alcoves $\x'$ and $\x$ both lie in the $J$-corridor, which is contained in the region between the hyperplanes $H_{\alpha,0}$ and $H_{\alpha,1}$ for all $\alpha \in \Phi_J^+$, the hyperplane $H$ is of the form $H = H_{\beta,k}$ for some $\beta \in \Phi^+ \setminus \Phi_J^+$ and $k \in \Z$.  Therefore the half-apartment $H^\x$ contains infinitely many alcoves of the form given by Lemma~\ref{lem:fixAlcoves}.  Choose one such alcove $t^{\lambda'}v'\fa$ and let $\gamma'$ be a minimal gallery from $t^{\lambda'}v'\fa$ to $\x$.  Let $\gamma''$ be the gallery from $t^{\lambda'}v'\fa$ to $\x'$ obtained by adding the alcove $\x'$ to the end of $\gamma'$.  Since $\gamma'$ is contained in $H^\x$, the minimal gallery $\gamma'$ does not cross $H$.  Thus $\gamma''$ is also a minimal gallery.  Now the first alcove $t^{\lambda'}v'\fa$ and last alcove $\x'$ of $\gamma''$ are both fixed by $I \cap L_J$, and $I \cap L_J$ takes minimal galleries to minimal galleries and preserves type.  Therefore $I \cap L_J$ fixes all alcoves in $\gamma''$.  Hence $I \cap L_J$ fixes $\x$, as required.

For (2), let $g \in U_J$, and let $\{ n_\beta \in \Z \mid \beta \in \Phi_J \setminus \Phi_J^+\}$ be such that, by Lemma~\ref{lem:fixAlcoves}(2), the element $g$ fixes all alcoves of the form $t^\lambda v\fa$ where \[\lambda \in \left( \bigcap_{\alpha \in \Phi_J^+} H_{\alpha,0} \right) \cap \left(\bigcap_{\beta \in \Phi^+ \setminus \Phi_J^+} \beta^{n_\beta,w_0}\right)\] and $v$ is a minimal length representative for a coset in $W_J \backslash \sW$.  Now all such alcoves lie in the intersection of half-apartments \[\left(\bigcap_{\alpha \in \Phi_J^+} \alpha^{0,\id} \cap \alpha^{1,w_0}\right) \cap  \left(\bigcap_{\beta \in \Phi^+ \setminus \Phi_J^+} \beta^{n_\beta,w_0} \right),\] and this intersection is a $J$-sector.  If $\x$ is any other alcove in this intersection, then $g$ fixes $\x$ by similar arguments to those used for (1).  Thus $g$ fixes some $J$-sector, as required.
\end{proof}

The following observation is immediate from our definition of a $(J,y)$-sector.

\begin{lemma}\label{lem:common}  The intersection of any two $(J,y)$-sectors contains a $(J,y)$-sector.
\end{lemma}

\begin{corollary}\label{cor:fixSector}  Let $g \in I_P^y$.  Then $g$ fixes some $(J,y)$-sector.
\end{corollary}
\begin{proof}  Since $I_P^y = (I \cap L_J)^y U_J^y$ where $P = P_J = L_J U_J$, the result follows from Corollary~\ref{cor:fixSector1} together with Lemma~\ref{lem:common}.
\end{proof}

The next pair of results formalize and generalize Proposition~11.2.4 of~\cite{GHKRadlvs}.

\begin{prop}\label{prop:inverse_chimney}  
Let $P = P_J$ be a standard spherical parabolic subgroup of $G$, and let $y \in \aW$.  Then for all $z \in \aW$, \[ r_{J,y}^{-1}(\z) = (I_P)^y z I.\]
\end{prop}

\begin{proof}  We may assume $y = \id$.  Let $h \in G$ and consider the coset $hI$, which we may identify with the alcove $h\fa$.  By Lemma~\ref{lem:Bruhat}, we can write $hI = gzI = g\z$ for some $g \in I_P$ and a unique $z \in \aW$.   Let $S_g$ be a $J$-sector which is fixed by $g$, as guaranteed by Corollary~\ref{cor:fixSector}.  Then the apartment $g(\App)$ contains the alcove $g\z$ and the $J$-sector $S_g$.  Hence by definition of chimney retractions, the restriction of $r_J$ to $g(\App)$ is the unique isomorphism from $g(\App)$ to $\App$ which fixes $g(\App) \cap \App$ pointwise.  Note that the retraction $r_J$ thus fixes the $J$-sector $S_g$. 

Now by Lemma~\ref{lem:bounded}, the image of $g\z$ under $r_J$ is equal to the image of $g\z$ under the retraction $r_d$, for some alcove $d$ (depending upon $g\z$) in some $J$-sector of $\App$.  By Lemma~\ref{lem:common}, we may assume that $d$ is in $S_g$.  Then $r_J(g\z)$ is the unique alcove, say $\z'$, in $\App$ such that there is a minimal gallery from $d$ to $\z'$ of the same type as a minimal gallery from $d$ to $g\z$.  Now the group element $g$ fixes $S_g$, hence $g$ fixes $d$, and $g$ takes minimal galleries to minimal galleries and preserves types.  Thus $g\z$ is the end alcove of a minimal gallery from $d$ to $g\z$ which is obtained by applying $g$ to a minimal gallery from $d$ to $\z$.  Thus $\z = \z'$, and hence $r_J^{-1}(z I) = I_P zI$, as required.
\end{proof}

\begin{corollary}\label{cor:inverse_chimney}  Let $P = P_J$ be a standard spherical parabolic subgroup of $G$.  For any faces $\sigma,\tau$ of $\fa$ which contain the origin and any elements $y, z \in \aW$, we have
\[
r_{J,y}^{-1}(z K(\tau)) = (I_P)^y z K (\tau).
\]
In particular, for all $\mu \in R^\vee$, letting $\sigma = \tau$ be the origin and $z = t^\mu$, we have \[ r_{J,y}^{-1}(\mu) = (I_P)^y t^\mu K. \]
\end{corollary}


\subsection{Shadows, double cosets, and labeled folded galleries}\label{sec:shadow-retract}

The goal of this final subection is to prove Theorems~\ref{thm:DoubleCosetsI-intro} and \ref{thm:DoubleCosetsK-intro}, as well as the more general Theorem \ref{thm:DoubleCosetsPar}. 

We first relate chimney shadows to certain double coset intersections in various affine flag varieties.  In particular, Theorem~\ref{thm:nonempty} below contains the trio of nonemptiness statements appearing in Theorems~\ref{thm:DoubleCosetsI-intro}, \ref{thm:DoubleCosetsK-intro}, and \ref{thm:DoubleCosetsPar}, respectively.  The following also generalizes Theorem~4.1 of~\cite{Hitzel} by the second author; see the introduction of that work for a discussion of related results.

\begin{thm}\label{thm:nonempty}  Let $P = P_J$ be a standard spherical parabolic subgroup of $G$, and let $y \in \aW$.  
\begin{enumerate}
\renewcommand{\labelenumi}{(\theenumi)}
\item \label{cor:nonemptyI} For all $x,z \in \aW$, we have 
\[
r_{J,y}(IxI) =\Sh_{J,y}(\x) 
\]
or equivalently
\[
\emptyset \neq IxI \cap (I_P)^y z I \iff \mathbf{z} \in \Sh_{J,y}(\mathbf{x}).
\]
\item  \label{cor:nonemptyK} 
For all $\lambda, \mu \in R^\vee$, we have 
\[
r_{J,y}(K t^\lambda K) =\Sh_{J,y}(\lambda) 
\]
or equivalently
\[
\emptyset \neq K t^\lambda K \cap (I_P)^y t^\mu K \iff \mu \in \Sh_{J,y}(\lambda).
\]
\item  \label{cor:nonemptyPar} 
For any faces $\sigma, \tau$ of $\fa$ which contain the origin and any $x, y, z \in \aW$, we have 
\[
r_{J,y}(K(\sigma) x K(\tau)) =\Sh_{J,y}(x\tau, \sigma) 
\]
or equivalently
\[
\emptyset \neq K(\sigma) x K(\tau) \cap (I_P)^y z K(\tau) \iff z\tau \in \Sh_{J,y}(x\tau, \sigma).
\]
\end{enumerate}
\end{thm}

\begin{proof}  This follows directly from Proposition~\ref{prop:shadowsRetractions}, Proposition~\ref{prop:inverse_classical}, Corollary~\ref{cor:inverse_classical}, Proposition~\ref{prop:inverse_chimney}, and Corollary~\ref{cor:inverse_chimney}.
\end{proof}

We are now prepared to prove the bijections occurring in Theorems~\ref{thm:DoubleCosetsI-intro} and \ref{thm:DoubleCosetsK-intro} from the introduction, as well as in the more general Theorem \ref{thm:DoubleCosetsPar} below.  

\begin{prop}\label{prop:DoubleCosetsI-body} 
  Let $x, y, z \in \aW$, and let $P = P_J$ be a standard spherical parabolic subgroup of $G$.  There is a bijection between the points of the intersection 
	\[  I x I  \cap (I_P)^y z I \] 
	and the set of $(J,y)$-labeled folded galleries of type $\vec{x}$ with final alcove $\z$.
\end{prop}

\begin{proof}
Recall from Proposition~\ref{prop:inverse_classical} that $r^{-1}(\x) = IxI$, and from Proposition~\ref{prop:inverse_chimney} that $ r_{J,y}^{-1}(\z) = (I_P)^y z I$.  The result now immediately follows from Theorem~\ref{thm:chimneyLFGs}.
\end{proof}

\begin{proof}[Proof of Theorem~\ref{thm:DoubleCosetsI-intro}]
Combine Theorem~\ref{thm:nonempty}(\ref{cor:nonemptyI}) and Proposition \ref{prop:DoubleCosetsI-body}.
\end{proof}

We next use Proposition~\ref{prop:DoubleCosetsI-body} to prove the more general Proposition \ref{prop:DoubleCosetsPar-body} below.  Given any faces $\sigma, \tau$ of $\fa$ which contain the origin $v_0$, we first require some additional terminology to identify certain representatives for the parahoric double coset $K(\sigma)x K(\tau)$.

\begin{definition} 
Let $\sigma, \tau$ be faces of $\fa$ which contain the origin $v_0$.  Then $x\in \aW$ is:
	\begin{enumerate}
		\item \textit{left-$W_\sigma$-reduced} if $\ell(wx)\geq \ell(x)$ for all $w \in W_\sigma$;
		\item \textit{right-$W_\tau$-reduced} if $\ell(xv)\geq \ell(x)$ for all $v \in W_\tau$; and
		\item \textit{$(W_\sigma,W_\tau)$-reduced} if $x$ is left-$W_\sigma$-reduced and right-$W_\tau$-reduced, i.e.~if $\ell(wx)\geq \ell(x)$ for all $w \in W_\sigma$ and $\ell(xv)\geq \ell(x)$ for all $v \in W_\tau.$
	\end{enumerate} 
\end{definition}

The hypothesis that $x$ be $(W_\sigma,W_\tau)$-reduced in the following proposition does not lead to a loss of generality, since we can always find a representative of the double coset $K(\sigma) x K(\tau)$ of this form. 

\begin{prop}\label{prop:DoubleCosetsPar-body} 
Let $P = P_J$ be a standard spherical parabolic subgroup of $G$. For any faces $\sigma, \tau$ of $\fa$ which contain the origin and any $x, y, z \in \aW$ such that $x$ is $(W_\sigma,W_\tau)$-reduced,
	there is a bijection between the points of the intersection 
	\[  K(\sigma) x K(\tau)  \cap (I_P)^y z K(\tau) \] 
	and the union over $w\in W_\sigma$ of the set of $(J,y)$-labeled folded galleries of type $\overrightarrow{wx}$ with final alcove $\z$. 
\end{prop}

\begin{proof}
The proof of the analogous result Theorem 5.14 from \cite{MNST} also holds in the setting of an arbitrary affine building.
\end{proof}

We can now state the most general version of the algebraic results in this paper.

\begin{thm}\label{thm:DoubleCosetsPar} 
Let $P=P_J$ be a standard spherical parabolic subgroup of $G$. For any faces $\sigma, \tau$ of $\fa$ which contain the origin and any $x, y, z \in \aW$ such that $x$ is $(W_\sigma,W_\tau)$-reduced,
	there is a bijection between the points of the intersection 
	\[  K(\sigma) x K(\tau)  \cap (I_P)^y z K(\tau) \] 
	and the union over $w\in W_\sigma$ of the set of $(J,y)$-labeled folded galleries of type $\overrightarrow{wx}$ with final alcove $\z$. Moreover, 
	 \[  K(\sigma) x K(\tau)  \cap (I_P)^y z K(\tau) \neq \emptyset \] if and only if the simplex $z\tau$ lies in the shadow $\Sh_{J,y}(x\tau,\sigma)$. 
\end{thm}

\begin{proof}
Combine Theorem~\ref{thm:nonempty}(\ref{cor:nonemptyPar}) and Proposition \ref{prop:DoubleCosetsPar-body}.
\end{proof}

We conclude by applying Proposition~\ref{prop:DoubleCosetsPar-body} to the special case where $K$ is the stabilizer of the origin $v_0$ in $G$.  Note by the Cartan decomposition that the dominance hypothesis on $\lambda$ does not lead to a loss of generality.

\begin{corollary}\label{cor:DoubleCosetsK-body}   Let $P = P_J$ be a standard spherical parabolic subgroup of $G$. Let $\lambda, \mu \in R^\vee$ with $\lambda$ dominant, and let $y \in \aW$.  Let $x_\lambda\fa$ be the unique alcove of $\App$ which contains $\lambda$ and is at minimal distance from $\fa$; equivalently, let $x_\lambda \in \aW$ be the unique minimal length representative of the coset $t^\lambda \sW$ in $\aW/\sW$.  
Then there is a bijection between the points of the intersection 
	\[Kt^\lambda K \cap (I_P)^y t^\mu K\] 
	and the union over $w\in \sW$ of the set of $(J,y)$-labeled folded galleries of type $\overrightarrow{wx_\lambda}$ ending in an alcove in $\star(\mu)$. 
\end{corollary}

\begin{proof}
The proof of the analogous result Corollary 5.18 from \cite{MNST} also holds in the setting of an arbitrary affine building.
\end{proof}

\begin{proof}[Proof of Theorem~\ref{thm:DoubleCosetsK-intro}]
Combine Theorem~\ref{thm:nonempty}(\ref{cor:nonemptyK}) and Corollary~\ref{cor:DoubleCosetsK-body}.
\end{proof}


\renewcommand{\refname}{Bibliography}
\bibliography{bibliographyIIG}
\bibliographystyle{alpha}

\end{document}